\documentclass[12pt,english]{amsart}
\usepackage[latin9]{inputenc}
\usepackage{babel}
\usepackage{amstext}
\usepackage{amsthm}
\usepackage{amssymb}
\usepackage[unicode=true,pdfusetitle,
 bookmarks=true,bookmarksnumbered=false,bookmarksopen=false,
 breaklinks=false,pdfborder={0 0 1},backref=false,colorlinks=false]
 {hyperref}

\makeatletter
\numberwithin{equation}{section}
\numberwithin{figure}{section}
\theoremstyle{plain}
\newtheorem{thm}{\protect\theoremname}[section]
  \theoremstyle{plain}
  \newtheorem{conjecture}[thm]{\protect\conjecturename}
  \theoremstyle{plain}
  \newtheorem{prop}[thm]{\protect\propositionname}
  \theoremstyle{plain}
  \newtheorem{lem}[thm]{\protect\lemmaname}

\usepackage{amsmath}
\usepackage{geometry}
\makeatother

  \providecommand{\conjecturename}{Conjecture}
  \providecommand{\lemmaname}{Lemma}
  \providecommand{\propositionname}{Proposition}
\providecommand{\theoremname}{Theorem}

\begin{document}

\title{The Distribution of $p$-Torsion in Degree $p$ Cyclic Fields}

\author{Jack Klys}
\begin{abstract}
We compute all the moments of the $p$-torsion in the first step of
a filtration of the class group defined by Gerth \cite{gerthprank}
for cyclic fields of degree $p$, unconditionally for $p=3$ and under
GRH in general. We show that it satisfies a distribution which Gerth
conjectured as an extension of the Cohen-Lenstra-Martinet conjectures.
In the $p=3$ case this gives the distribution of the $3$-torsion
of the class group modulo the Galois invariant part. We follow the
strategy used by Fouvry and Kl$\ddot{\text{u}}$ners in their proof
of the distribution of the $4$-torsion in quadratic fields \cite{fk1}. 
\end{abstract}

\maketitle

\section{Introduction}

Let $K$ be a number field of degree $n$. Let $Cl_{K}$ denote the
class group and $Cl_{K,p}$ denote the $p$-part. Let $S$ be the
set of finite abelian $p$-groups. We are interesed in the question:
what is the probability of any $A\in S$ occuring as $Cl_{K,p}$ for
$K$ of degree $n$? The Cohen-Lenstra heuristics propose an answer
to this question in the form of a probability distribution on $S$.

We make the question more precise as follows. Let $D_{K}$ denote
the discriminant of $K$. For any $X$ define 
\[
S_{X}^{\pm}\left(A\right)=\frac{\left|\left\{ K\mid0<\pm D_{K}<X,Cl{}_{K,p}\cong A\right\} \right|}{\sum_{K,0<\pm D_{K}<X}1}.
\]
 The probability of $A$ occuring as $Cl_{K,p}$ is $\lim_{X\longrightarrow\infty}S_{X}\left(A\right)$.
In general this is not known to exist. Cohen and Lenstra conjectured
\cite{cohenlenstra} that it does and proposed a distribution on $S$
which should equal this quantity. For $s\in\mathbb{Z}_{\ge1}\cup\left\{ \infty\right\} $
let
\[
\eta_{s}\left(p\right)=\prod_{i=1}^{s}\left(1-1/p^{i}\right).
\]
One can show \cite{cohenlenstra,hall} that  
\[
\sum_{G\in S}\frac{1}{\left|\mathrm{Aut}G\right|}=\frac{1}{\eta_{\infty}\left(p\right)}<\infty.
\]
Then for any $A\in S$ and $u\ge0$ define 
\[
\mu_{u}\left(A\right)=\frac{\eta_{\infty}\left(p\right)}{\left|\mathrm{Aut}A\right|\left|A\right|^{u}}.
\]
This defines a probability measure on $S$, called the Cohen-Lenstra
distribution. They originally considered the case $n=2$, $p\neq2$
and considered complex quadratic and real quadratic fields seperately.
\begin{conjecture}[Cohen-Lenstra]
For $A\in S$

\[
\mu_{0}\left(A\right)=\lim_{X\longrightarrow\infty}S_{X}^{-}\left(A\right)
\]
\[
\mu_{1}\left(A\right)=\lim_{X\longrightarrow\infty}S_{X}^{+}\left(A\right).
\]
\end{conjecture}
These conjectures were extended to higher degree number fields by
Cohen and Martinet \cite{cohenmartinet} again for $p\nmid n$. 

No cases of these conjectures are known in full strength, though there
has been much recent work on the subject. In the setting of number
fields there are results giving the average size of the class group
or sugroup therof. There is the classical result of Davenport-Heilbronn
\cite{davenportHeilbronn} and Datskovsky-Wright \cite{datskovskyWright}
for the average size of 3-torsion of quadratic fields. Bhargava-Varma
\cite{BhargavaVarma} compute the average 2 torsion of cubic fields.
There are also partial results for 8 and 16 torsion of quadratic fields
due to Milovic \cite{milovic16rank,milovic8rank} and Smith \cite{alexandersmith}.
Below we will discuss in more detail the work of Fouvry and Kl$\ddot{\text{u}}$ners
\cite{fk1} on 4 torsion of quadratic fields. There are also non-abelian
versions which have been studied by Alberts \cite{alberts} and Bhargava
\cite{BhargavaNonabelian}.

The conjectures have also been studied in the setting of function
fields which provides additional tools such as moduli schemes. Some
results here are the work of Ellenberg-Venkatesh-Westerland \cite{EVW},
Boston-Wood \cite{bostonwood} and Wood \cite{Wood}.

The original conjectures ignored the case when $p$ divides the degree
of the number fields. Gerth extended them to $p$-torsion in degree
$p$ cyclic fields by proposing a distribution for a certain subgroup
of $Cl_{K}\left[p\right]$ \cite{gerth4rankproofs,gerthprank,gerth4rankidea}.
For the case $p=2$ and $n=2$ Gerth proposed the conjectures should
hold in their original form, but with $Cl_{K}^{2}$ instead of $Cl_{K}$.
This was proved by Fouvry and Kl$\ddot{\text{u}}$ners \cite{fk1}.
To state their result, let $\mathrm{rk}_{4}\left(Cl_{K}\right)=\mathrm{rk}_{2}\left(Cl_{K}^{2}\right)$
and for any $k\in\mathbb{Z}_{\ge1}$ let

\[
M_{k}^{\pm}\left(2\right)=\lim_{X\longrightarrow\infty}\frac{\sum_{K,0<\pm D_{K}<X}2^{k\mathrm{rk}_{4}\left(Cl_{K}\right)}}{\sum_{K,0<\pm D_{K}<X}1}.
\]
 Define $N\left(k,p\right)$ to be the number of subspaces of $\mathbb{F}_{p}^{k}$.
\begin{thm}[Fouvry-Kl$\ddot{\text{u}}$ners]
\label{thm:fouvrykluners1}For every $k\in\mathbb{Z}_{\ge1}$

\begin{eqnarray*}
M_{k}^{-}\left(2\right) & = & N\left(k,2\right)\\
M_{k}^{+}\left(2\right) & = & N\left(k+1,2\right)-N\left(k,2\right).
\end{eqnarray*}
\end{thm}
By a separate result \cite{fk2distribution} Fouvry and Kl$\ddot{\text{u}}$ners
deduce that these moments are enough to determine a distribution
\begin{thm}[Fouvry-Kl$\ddot{\text{u}}$ners]
The density of complex quadratic fields $K$ with $\mathrm{rk}_{4}\left(Cl_{K}\right)=s$
is 
\[
\frac{\eta_{\infty}\left(2\right)}{\eta_{s}^{2}\left(2\right)2^{s^{2}}}
\]
and the density of real quadratic fields with $\mathrm{rk}_{4}\left(Cl_{K}\right)=s$
is 
\[
\frac{\eta_{\infty}\left(2\right)}{\eta_{s}\left(2\right)\eta_{s+1}\left(2\right)2^{s\left(s+1\right)}}.
\]
 
\end{thm}
Gerth conjectured a distribution for a certain subgroup of $Cl_{K}\left[p\right]$
of cyclic $p$ fields for all $p$. To state it we first define some
notation. Let $K$ be a cyclic field of degree $p$ with Galois group
$G=\left\langle \sigma\right\rangle $. Let $\varphi=1-\sigma$ act
on $Cl_{K}\left[p\right]$. It can be shown there is a filtration
\[
Cl_{K}\left[p\right]^{G}=\ker\varphi\subseteq\ker\varphi^{2}\subseteq\cdots\subseteq\ker\varphi^{p-1}=Cl_{K}\left[p\right].
\]
Then Gerth conjectured a distribution for the $p$-rank of $\varphi\left(\ker\varphi^{2}\right)$.
Notice that for $p=3$ we have $\ker\varphi^{2}=Cl_{K}\left[3\right]$
and so the above filtration implies $\varphi\left(\ker\varphi^{2}\right)\cong Cl_{K}\left[3\right]/Cl_{K}\left[3\right]^{G}$.
We prove the following theorem which verifies Gerth's conjecture for
$p=3$:
\begin{thm}
\label{thm:gerth 3 rank}The density of cyclic cubic fields with $\mathrm{rk}_{3}\left(Cl_{K}\left[3\right]/Cl_{K}\left[3\right]^{G}\right)=s$
is 
\[
\frac{\eta_{\infty}\left(3\right)}{\eta_{s}\left(3\right)\eta_{s+1}\left(3\right)3^{s\left(s+1\right)}}.
\]
\end{thm}
We can extend this to all $p$ under the assumption of GRH:
\begin{thm}
\label{thm:gerth p ranks}Assume GRH for Artin $L$-functions. The
density of degree $p$ cyclic fields with $\mathrm{rk}_{p}\left(\varphi\left(\ker\varphi^{2}\right)\right)=s$
is 
\[
\frac{\eta_{\infty}\left(p\right)}{\eta_{s}\left(p\right)\eta_{s+1}\left(p\right)p^{s\left(s+1\right)}}.
\]
\end{thm}
Before continuing we make some remarks about $Cl_{K}\left[p\right]^{G}$.
It is the part of $Cl_{K}\left[p\right]$ corresponding by class field
theory to the genus field of $K$, that is the maximal unramified
extension of $K$ which is abelian over $\mathbb{Q}$. It can be shown
$\left|Cl_{K}\left[p\right]^{G}\right|=p^{r-1}$ where $r$ is the
number of primes ramified in $K$ and that the average of $\mathrm{rk}_{p}\left(Cl_{K}\left[p\right]^{G}\right)$
is $\infty$. In the case $p=2$ removing this part corresponds to
replacing $2$-rank by $4$-rank.

We deduce Theorems \ref{thm:gerth 3 rank} and \ref{thm:gerth p ranks}
from the following theorem together with \cite{fk2distribution}.
Define 
\[
M_{k}\left(p\right)=\lim_{X\longrightarrow\infty}\frac{\sum_{K,D_{K}<X}p^{k\mathrm{rk}_{p}\left(\varphi\left(\ker\varphi^{2}\right)\right)}}{\sum_{K,D_{K}<X}1}.
\]

\begin{thm}
\label{thm:mainthm}Let $k\in\mathbb{Z}_{\ge1}$. Then unconditionally
for $p=3$ and under the assumption of GRH for Artin $L$-functions
for $p>3$ we have 
\[
M_{k}\left(p\right)=N\left(k+1,p\right)-N\left(k,p\right).
\]
\end{thm}
The proof of Theorem \ref{thm:mainthm} follows the strategy of Fouvry
and Kl$\ddot{\text{u}}$ners. For any degree $p$ cyclic field $K$
we express $\left|Cl_{K}\left[p\right]\right|$ using a sum of idele
class characters, and then sum over all degree $p$ cyclic fields
of discriminant up to $X$. We then study the asymptotics of this
expression using techniques from analytic number theory. In the $p=3$
case we require several versions of a large seive inequality for cubic
characters to bound the error term. We prove one such version as well
as applying several others from the literature, due to Heath-Brown
\cite{heathbrowncubicguasssums}, Baier-Young \cite{StephanBaierMatthewYoung}
and Iwaniec-Kowalski \cite{IwaniecKowalski}. The reason for assuming
GRH in the general case is that certain versions of the large seive
are not yet available for order $p$ characters. In particular we
lack analogues of Propositions \ref{prop:bs cubic epsilon } and \ref{prop:bs dirichlet epsilon}.
This is the only obstacle to an unconditional proof for all $p$.

Finally we remark briefly about an equivalent formulation of the Cohen-Lenstra
conjectures which is commonly used. The distribution $\mu_{u}$ is
characterized by the fact \cite{EVW} that for all $A\in S$
\begin{equation}
\mathbb{E}_{G\sim\mu_{u}}\left(\left|\mathrm{Sur}\left(G,A\right)\right|\right)=\sum_{G\in S}\mu_{u}\left(G\right)\cdot\left|\mathrm{Sur}\left(G,A\right)\right|=\frac{1}{\left|A\right|^{u}}.\label{eq:cld characterization}
\end{equation}
 This is often called the $A$-moment of $\mu_{u}$ and computing
it only for certain $A$ can still provide information about the distribution
of elements in $Cl_{K}$. 

It is clear that $\left|\mathrm{Hom}\left(G,\left(\mathbb{Z}/p\mathbb{Z}\right)^{k}\right)\right|=p^{k\mathrm{rk}_{p}\left(G\right)}$.
Furthermore $\left|\mathrm{Hom}\left(G,\left(\mathbb{Z}/p\mathbb{Z}\right)^{k}\right)\right|=\sum_{i=0}^{k}n\left(k,i,p\right)\left|\mathrm{Sur}\left(G,\left(\mathbb{Z}/p\mathbb{Z}\right)^{i}\right)\right|$
where $n\left(k,i,p\right)$ is the number of $i$-dimensional subspaces
of $\mathbb{F}_{p}^{k}$. Hence theorem \ref{thm:mainthm} can be
rephrased as computing the $A$ moments in the above sense for all
the groups $A=\left(\mathbb{Z}/3\mathbb{Z}\right)^{k}$.

\subsection*{Acknowledgements}

The author extends his thanks to several people for helpful and stimulating
discussions while working on this problem: John Friedlander, J$\ddot{\text{u}}$rgen
Kl$\ddot{\text{u}}$ners, Jacob Tsimerman, Asif Zaman. He would also
like to thank Leo Goldmakher and Matthew Young for providing useful
references on character sums.

\section{Counting $p$-torsion in degree $p$ cyclic fields}

Let $K$ be a degree $p$ Galois extension of $\mathbb{Q}$ with Galois
group $G=\left\langle \sigma\right\rangle $. Let $\varphi=1-\sigma$.
There is a filtration 
\[
Cl_{K}\left[p\right]^{G}=\ker\varphi\subseteq\ker\varphi^{2}\subseteq\cdots\subseteq\ker\varphi^{p-1}=Cl_{K}\left[p\right].
\]
 From this we can write down the exact sequence
\[
1\longrightarrow Cl_{K}\left[p\right]^{G}\longrightarrow\ker\varphi^{2}\longrightarrow\varphi\left(\ker\varphi^{2}\right)\longrightarrow1
\]
 so that $\left|\ker\varphi^{2}\right|=\left|Cl_{K}^{G}\right|\left|\varphi\left(\ker\varphi^{2}\right)\right|$.
We will consider $\varphi$ acting on $\ker\varphi^{2}$ and throughout
the section write $\mathrm{im}\varphi=\varphi\left(\ker\varphi^{2}\right)$.

Denote by $N$ the norm map $N_{K/\mathbb{Q}}$ (both on ideals and
elements of $K$). Let $J$ be the group of fractional ideals of $K$.
Furthermore let $P_{1},\ldots,P_{r}$ be the ramified primes of $K$,
and let $B=\left\{ P_{1}^{e_{k}}\cdots P_{r}^{e_{r}}\mid e_{i}=0,1,\ldots,p-1\right\} $. 
\begin{prop}
Let $r$ be the number of primes ramified in $K$. Then $\left|Cl_{K}^{G}\right|=p^{r-1}$. 
\end{prop}
\begin{proof}
Let $\phi:\left(\mathbb{Z}/p\mathbb{Z}\right)^{r}\longrightarrow Cl_{K}^{G}$
be the homomorphism $\left(A_{1},\ldots,A_{r}\right)\longmapsto\left[\prod_{i}P_{i}^{a_{i}}\right]$.
We will show that the $P_{i}$ generate $Cl_{K}{}^{G}$ which will
imply that $\phi$ is surjective. Suppose $\left[I\right]\in Cl_{K}{}^{G}$,
so that $\left(\alpha\right)I=\sigma\left(I\right)$ for some $\alpha\in K^{\times}$.
Taking norm gives $N\left(\alpha\right)=1$ so we can pick $\alpha$
such that $N\alpha=1$ in $K$. By Hilbert's Theorem 90 there is some
$\beta\in K$ such that $\alpha=\left(1-\sigma\right)\beta$. Thus
we get the equality of ideals $\left(\beta\right)I=\sigma\left(\left(\beta\right)I\right)$.
This implies $\left(\beta\right)I$ is supported only on ramified
and rational primes. Thus $\left[I\right]\in\mathrm{im}\phi$ and
hence $\phi$ is surjective.

It remains to show that $\ker\phi$ is non-empty. Let $h:\mathcal{O}_{K}^{\times}/\pm1\longrightarrow\mathcal{O}_{K}^{\times}/\pm1$
be defined by $h\left(u\right)=u/\sigma\left(u\right)$. First we
show $\left|\mbox{coker}h\right|=p$. Consider the map $\log_{K}:K\longrightarrow\prod_{i}\mathbb{R}_{\sigma^{i}}$.
By Dirichlet's unit theorem $\log_{K}$ gives an isomorphism between
$\mathcal{O}_{K}^{\times}/\pm1$ and a lattice of rank $p-1$ contained
in the trace 0 hyperplane. Let $V$ be this lattice. Then $1-\sigma$
acts as a linear map on $V$ and has eigenvalues $1-\zeta^{i}$ for
$1\le i\le p-1$ and hence has determinant $p$. Thus the index of
$\left(1-\sigma\right)\log_{K}\left(\mathcal{O}_{K}^{\times}/\pm1\right)=\log_{K}\left(\left(1-\sigma\right)\left(\mathcal{O}_{K}^{\times}/\pm1\right)\right)$
in $\log_{K}\left(\mathcal{O}_{K}^{\times}/\pm1\right)$ is $p$.
Thus $\left|\mbox{coker}h\right|=p$.

Let $u\in\mathcal{O}_{K}^{\times}$ represent a non-trivial element
of coker$h$. Since $u\in\mathcal{O}_{K}^{\times}$ we have $N_{K/\mathbb{Q}}u=1$
so by Hilbert's Theorem 90 $u=\sigma\left(x\right)/x$ for some $x\in K$.
Since $u$ is non-trivial in $\mbox{coker}h$ this implies $x\notin\mathcal{O}_{K}^{\times}$.
So consider $I=\left(x\right)$ which is not the unit ideal. Since
$\sigma\left(x\right)/x\in\mathcal{O}_{K}^{\times}$ this implies
$I$ is supported only on ramified primes and rational primes. If
$x$ is only supported on rational primes then $u=\sigma\left(x\right)/x=1$
which contradicts the fact that $u$ represent a non-trivial element
of coker$h$. This shows that $\ker\phi$ is non-empty, which proves
what we wanted.
\end{proof}
Next we give another description of $\mathrm{im}\varphi$. 
\begin{lem}
\label{lem:JtrivialH1}Consider $N$ acting on $J$ the group of fractional
ideals of $K$. Then 
\[
\ker N/\mathrm{im}\varphi=1.
\]
 
\end{lem}
\begin{proof}
It is clear that $\mathrm{im}\left(1-\sigma\right)\subset\ker N$.
Suppose $NI=1$ for some ideal $I\in J$. Then $I$ can only be supported
on split primes. Any set of $p$ split primes $Q_{1}^{a_{1}}Q_{2}^{a_{2}}\cdots Q_{p}^{a_{p}}$
in the decomposition of $I$ contributes $q^{\sum a_{i}}$ in $NI$,
which implies that $\sum a_{i}=0$. Then $Q_{1}^{a_{1}}Q_{2}^{a_{2}}\cdots Q_{p}^{a_{p}}=\left(Q_{1}^{a_{1}}Q_{2}^{a_{1}+a_{2}}\cdots Q_{p-1}^{a_{1}+\cdots+a_{p-1}}\right)^{1-\sigma}$.
Applying this to all primes in $I$ shows $I\in\mathrm{im}\left(1-\sigma\right)$.
\end{proof}
\begin{lem}
\label{lem:im(1 - sig) equals norm(alpha)}For any class $\mathfrak{b}\in Cl_{K}{}^{G}$
we have 
\[
\mathfrak{b}\in\mathrm{im}\varphi\iff N\mathfrak{b}=N\alpha
\]
 for some $\alpha\in K$ (note this condition is independent of the
ideal representing $\mathfrak{b}$). 
\end{lem}
\begin{proof}
Suppose first that $\mathfrak{b}\in\mathrm{im}\varphi$. So for some
$\mathfrak{a}\in Cl_{K}$, we have $\mathfrak{b}=\left(1-\sigma\right)\left(\mathfrak{a}\right)=\mathfrak{a}/\sigma\left(\mathfrak{a}\right)$.
Let $\mathfrak{a}$ also denote some representative such that $\alpha\mathfrak{b}=\mathfrak{a}/\sigma\left(\mathfrak{a}\right)$
in the group of ideals. Taking norm of this gives $N\mathfrak{b}=N\alpha^{-1}$
which proves one direction.

Now suppose that $N\mathfrak{b}=N\alpha$ for some $\alpha\in K$.
Hence $\mathfrak{b}=\left(\alpha\right)I$ for some ideal $I\in\ker N$.
By Lemma \ref{lem:JtrivialH1} we have $I=\mathrm{im}\left(1-\sigma\right)\mathfrak{a}$
for some ideal $\mathfrak{a}\in\ker\left(1-\sigma\right)^{2}$. 

It remains to show $\mathfrak{a}\in Cl_{K}\left[p\right]$. Since
$\mathfrak{b}=\mathfrak{a}^{1-\sigma}$ and $\mathfrak{b}\in Cl_{K}{}^{G}$
this implies $\mathfrak{a}^{1-\sigma}=\mathfrak{a}^{\sigma-\sigma^{2}}$,
and so $\mathfrak{a}\mathfrak{a}^{\sigma^{2}}=\mathfrak{a}^{\sigma}\mathfrak{a}^{\sigma}$.
Multiplying both sides by $\mathfrak{a}^{\left(p-2\right)\sigma}$
gives
\begin{eqnarray*}
\left(\mathfrak{a}^{\sigma}\right)^{p} & = & \mathfrak{a}^{1+\left(p-2\right)\sigma+\sigma^{2}}\\
 & = & \mathfrak{a}^{\left(1-\sigma\right)^{2}}\\
 & = & 1
\end{eqnarray*}
Thus $\mathfrak{a}\in Cl_{K}\left[p\right]$ as required.
\end{proof}
Let $D_{K}$ denote the discriminant of $K$. Then the discriminant
is of the form $D_{K}=\left(p_{1}\cdots p_{r}\right)^{p-1}$ where
each factor is a prime congruent to 1 mod $p$ or equal to $p^{2}$
and they are distinct. Now let 
\[
D=\begin{cases}
D_{K} & \mbox{if }p\nmid D_{K}\\
D_{K}/p^{2} & \mbox{if }p\mid D_{K}.
\end{cases}
\]
 Let $\mathfrak{b}\in Cl_{K}{}^{G}$ and $b=N\mathfrak{b}$. The class
$\mathfrak{b}$ has a representative which lies in $B$ which implies
$b\mid D$. By Lemma \ref{lem:im(1 - sig) equals norm(alpha)} we
want to count the number of classes $\mathfrak{b}\in Cl_{K}{}^{G}$
such that $N\mathfrak{b}=N\alpha$ for some $\alpha\in K$. This is
the number of divisors $b$ of $D$ such that $b=N\alpha$ for some
$\alpha\in K$, divided by $p$, since $p$ different divisors will
come from the same class in $Cl_{K}{}^{G}$.

Thus we have shown 
\begin{prop}
\label{prop: im(1-sigma) equals divisors of D}With the above notation
\end{prop}
\[
\left|\mathrm{im}\varphi\right|=\frac{1}{p}\left\{ b\mid D\mid b=N\left(\alpha\right)\mbox{ for some }\alpha\in K^{\times}\right\} .
\]

\section{The $p$-torsion as a character sum}

Let $K/\mathbb{Q}$ be a degree $p$ cyclic extension. Then the discriminant
is of the form $D_{K}=\left(p_{1}\cdots p_{r}\right)^{p-1}$ where
each factor is a prime congruent to 1 mod $p$ or equal to $p^{2}$
and they are distinct. Conversely every integer of this form is a
discriminant of a degree $p$ cyclic field. Each such extension corresponds
to a character $\chi$ of $C_{\mathbb{Q}}\cong\prod_{l}\mathbb{Z}_{l}^{\times}$
the idele class group of $\mathbb{Q}$ with $\mathrm{ker}\chi=NC_{K}$
an index $p$ subgroup of $C_{\mathbb{Q}}$. That is, a character
$\chi:\left(1+p\mathbb{Z}_{p}\right)\times\prod_{l\mid D_{K},l\neq p}\mathbb{F}_{l}^{\times}\longrightarrow\mu_{p}$
where the $\left(1+p\mathbb{Z}_{p}\right)$ factor appears only when
$p\mid D_{K}$. The character $\chi$ is non-trivial on each factor.
Furthemore it factors into local components $\chi=\prod_{l\mid D_{K}}\chi_{l}$.

The goal of this section will be to prove 
\begin{thm}
\label{thm:im(1-sigma) is a character sum}For each degree $p$ cyclic
field $K$ let $\sigma_{K}$ denote a generator of the Galois group
and $D_{K}$ the discriminant. Then 

\begin{equation}
\sum_{K,D_{K}<X^{p-1}}\left|\mathrm{im}\left(1-\sigma_{K}\right)\right|=\frac{1}{\left(p-1\right)p}\sum_{D<X}\mu^{2}\left(D\right)\frac{1}{p^{\omega\left(D\right)}}\sum_{D_{0}\cdots D_{p^{2}-1}}\sum_{\left(\chi_{l}\right)}\prod_{v}\prod_{l\mid D_{v_{1}p+v_{2}}}\chi_{l}\left(\prod_{u}D_{u_{1}p+u_{2}}^{\Phi\left(u,v\right)}\right)\label{eq:sum im 1-s}
\end{equation}

where on the right the first sum is over square-free integers whose
prime factors are congruent to $1$ mod $p$ or equal to $p$ and
the second sum is over all such factorizations of $D$ and the third
sum is over all tuples of order $p$ characters $\left(\chi_{p}\right)_{p\mid D}$.
The products are over $u,v\in\mathbb{Z}/p^{2}\mathbb{Z}$, and \textup{$\Phi\left(u_{1}p+u_{2},v_{1}p+v_{2}\right)=u_{1}\left(v_{2}-u_{1}\right)$.}
\end{thm}
\begin{proof}
First fix $K$. We start with Proposition \ref{prop: im(1-sigma) equals divisors of D}.
Let $D$ be as defined at the end of the previous section. Since $K$
is cyclic, by the Hasse norm theorem $b$ is a global norm if and
only if $b$ is a local norm everywhere: 
\[
b=N\alpha\mbox{ for some \ensuremath{\alpha\in K}}\iff b=N\alpha_{p}\mbox{ for some \ensuremath{\alpha_{p}\in K_{p}}, for all \ensuremath{p}}.
\]
 Here $K_{p}=K\otimes_{\mathbb{Q}}\mathbb{Q}_{p}$. So we want to
detect when $b\mid D$ is a norm in $K_{p}\cong K\otimes\mathbb{Q}_{p}$
for all $p\mid D$. 

It is a standard fact that $b$ is a local norm at $l$ if and only
if the idele $\left(1,\ldots,b,\ldots,1\right)\in C_{\mathbb{Q}}$
is a global norm from $C_{K}$ (here $b$ is in the coordinate corresponding
to $l$). Under the identification $C_{\mathbb{Q}}\cong\prod\mathbb{Z}_{l}^{\times}$
this is the idele $i_{b,l}=\left(\frac{1}{l^{i}},\ldots,\frac{b}{l^{i}},\ldots,\frac{1}{l^{i}}\right)$,
where $i$ is maximal such that $l^{i}\mid b$. Recall that $NC_{K}=\mathrm{ker}\chi$.
Hence $b$ is a local norm at $l$ if and only if $i_{b,l}$ is in
the kernel of the character $\chi$ as described above: 
\[
\chi\left(i_{b,l}\right)=\chi_{l}\left(\frac{b}{l^{i}}\right)\prod_{q\neq l}\chi_{q}\left(\frac{1}{l^{i}}\right)=1.
\]
We write $\chi\left(b,l\right)=\chi\left(i_{b,l}\right)$. Hence we
will use the following expression to detect when $b$ is a local norm:
\[
\left(\frac{1+\chi+\cdots+\chi^{p-1}}{p}\right)\left(b,l\right)=\begin{cases}
1 & \mbox{if \ensuremath{b} is a norm at \ensuremath{l}}\\
0 & \mbox{else}
\end{cases}.
\]
Note that $D$ is a $p-1$ power. Write a divisor of $D$ as $b_{1}b_{2}^{2}\cdots b_{p-1}^{p-1}$
where the $b_{i}$ are square free. Thus 
\[
p\left|\mathrm{im}\left(1-\sigma_{K}\right)\right|=\sum_{b_{1}b_{2}^{2}\cdots b_{p-1}^{p-1}\mid D}\left\{ b_{1}b_{2}^{2}\cdots b_{p-1}^{p-1}\mbox{ is a global norm}\right\} 
\]
In the following we will let $B=b_{1}b_{2}^{2}\cdots b_{p-1}^{p-1}$.
Using the above we get that this is equal to 

\begin{eqnarray*}
\\
\sum_{D=\left(b_{0}\cdots b_{p-1}\right)^{p-1}}\prod_{l\mid D}\left(\frac{1+\chi+\cdots+\chi^{p-1}}{p}\right)\left(B,l\right)\\
=\frac{1}{p^{\omega\left(D\right)}}\sum_{D=\left(b_{0}\cdots b_{p-1}\right)^{p-1}}\prod_{i}\prod_{l\mid b_{i}}\left(1+\cdots+\chi^{p-1}\right)\left(B,l\right)
\end{eqnarray*}
Further expanding gives

\begin{eqnarray*}
\frac{1}{p^{\omega\left(D\right)}}\sum_{D=\left(b_{0}\cdots b_{p-1}\right)^{p-1}}\prod_{i=0}^{p-1}\left(\sum_{b_{i}=D_{ip}D_{ip+1}\cdots D_{ip+p-1}}\prod_{j=0}^{p-1}\prod_{p\mid D_{ip+j}}\chi\left(B^{j},l\right)\right).
\end{eqnarray*}
By definition of $\chi\left(B^{j},l\right)$ this can be rewritten
as 
\[
\frac{1}{p^{\omega\left(D\right)}}\sum_{D_{0}\cdots D_{p^{2}-1}}\prod_{i=0}^{p-1}\prod_{j=0}^{p-1}\prod_{l\mid D_{ip+j}}\prod_{q\mid D,q\neq l}\chi_{q}\left(\frac{1}{l^{ij}}\right)\chi_{l}\left(\frac{B^{j}}{l^{ij}}\right).
\]
 where the first sum is over all such factorizations into coprime
integers of $D^{1/p-1}$. Further rearranging:

\[
\frac{1}{p^{\omega\left(D\right)}}\sum_{D_{0}\cdots D_{p^{2}-1}}\prod_{i=0}^{p-1}\prod_{j=0}^{p-1}\left[\prod_{q\mid D/D_{ip+j}}\chi_{q}\left(\prod_{l\mid D_{ip+j}}\frac{1}{l^{ij}}\right)\right]\left[\prod_{q\mid D_{ip+j}}\chi_{q}\left(\prod_{l\mid D_{ip+j}/q}\frac{1}{l^{ij}}\right)\right]\left[\prod_{l\mid D_{ip+j}}\chi_{l}\left(\frac{B^{j}}{l^{ij}}\right)\right]
\]

\[
=\frac{1}{p^{\omega\left(D\right)}}\sum_{D_{0}\cdots D_{p^{2}-1}}\prod_{i=0}^{p-1}\prod_{j=0}^{p-1}\left[\prod_{q\mid D/D_{ip+j}}\chi_{q}\left(\frac{1}{D_{ip+j}^{ij}}\right)\right]\left[\prod_{q\mid D_{ip+j}}\chi_{q}\left(\frac{q^{ij}}{D_{ip+j}^{ij}}\right)\right]\left[\prod_{q\mid D_{ip+j}}\chi_{q}\left(\frac{B^{j}}{q^{ij}}\right)\right].
\]
Now let $\overline{D_{j}}=\prod_{i=0}^{p-1}D_{ip+j}$ and $\overline{D}=\prod_{i,j}D_{ip+j}^{ij}$.
Then the above is 

\begin{eqnarray*}
\frac{1}{p^{\omega\left(D\right)}}\sum_{D_{0}\cdots D_{p^{2}-1}}\prod_{j=0}^{p-1}\prod_{l\mid\overline{D_{j}}}\chi_{l}\left(\frac{B^{j}}{\overline{D}}\right).
\end{eqnarray*}
From the definition of $B$ we have 
\[
B=\prod_{i=0}^{p-1}\left(D_{ip}D_{ip+1}\cdots D_{ip+p-1}\right)^{i}
\]
hence the exponent of $D_{u_{1}p+u_{2}}$ in $B^{v_{2}}/\overline{D}$
is 
\[
u_{1}v_{2}-u_{1}u_{2}=u_{1}\left(v_{2}-u_{2}\right).
\]
 Let $u=u_{1}p+u_{2}$ and $v=v_{1}p+v_{2}$. Define a map on $\mathbb{Z}/p^{2}\mathbb{Z}$
\[
\Phi\left(u,v\right)=u_{1}\left(v_{2}-u_{2}\right).
\]
 Thus we conclude that 
\begin{equation}
\left|\mathrm{im}\left(1-\sigma_{K}\right)\right|=\frac{1}{p}\frac{1}{p^{\omega\left(D\right)}}\sum_{D_{0}\cdots D_{p^{2}-1}}\prod_{v}\prod_{l\mid D_{v_{1}p+v_{2}}}\chi_{l}\left(\prod_{u}D_{u}^{\Phi\left(u,v\right)}\right).\label{eq:expthistok-1}
\end{equation}

Summing over all characters corresponds to summing over all degree
$p$ cyclic fields of discriminant $D$ but overcounts by a factor
of $p-1$ since for a fixed discriminant $D$ the characters $\prod_{l\mid D}\chi_{l}$
and $\prod_{l\mid D}\chi_{l}^{j}$ for $0\le j\le p-1$ correspond
to the same field (indeed the above expression does not change if
we replace each $\chi_{l}$ with $\chi_{l}^{j}$ simultaneously).
Thus we have shown 

\[
\sum_{K,D_{K}=D}\left|\mathrm{im}\left(1-\sigma_{K}\right)\right|=\frac{1}{\left(p-1\right)p}\frac{1}{p^{\omega\left(D\right)}}\sum_{D_{0}\cdots D_{p^{2}-1}}\sum_{\left(\chi_{l}\right)}\prod_{v}\prod_{l\mid D_{v}}\chi_{l}\left(\prod_{u}D_{u}^{\Phi\left(u,v\right)}\right).
\]
where the second sum on the right is over all tuples of characters
$\left(\chi_{p}\right)_{p\mid D}$. Since we are interested in computing
the average over all degree $p$ Galois fields we sum over these (up
to $X^{p-1}$) to get 

\[
\sum_{K,D_{K}<X^{p-1}}\left|\mathrm{im}\left(1-\sigma_{K}\right)\right|=\frac{1}{\left(p-1\right)p}\sum_{D<X}\mu^{2}\left(D\right)\frac{1}{p^{\omega\left(D\right)}}\sum_{D_{0}\cdots D_{p^{2}-1}}\sum_{\left(\chi_{l}\right)}\prod_{v}\prod_{l\mid D_{v}}\chi_{l}\left(\prod_{u}D_{u}^{\Phi\left(u,v\right)}\right).
\]
which proves Theorem \ref{thm:im(1-sigma) is a character sum}. 
\end{proof}

\section{\label{sec:expression for kth moment}An expression for the $k$th
moment}

Define 
\[
S_{k}\left(X\right)=\sum_{K,D_{K}<X^{p-1}}\left|\mathrm{im}\left(1-\sigma_{K}\right)\right|^{k}.
\]
 We want to generalize Theorem \ref{thm:im(1-sigma) is a character sum}
to obtain a similar expression for $S_{k}\left(X\right)$.

We follow the same method as Fouvry and Kl$\ddot{\text{u}}$ners,
to write the $k$ factorizations of $D$ as
\[
D=\prod_{u_{1}}D_{u_{1}}^{\left(1\right)}=\cdots=\prod_{u_{k}}D_{u_{k}}^{\left(k\right)}
\]
 where each index $u_{i}\in\mathbb{Z}/p^{2}\mathbb{Z}$ (note this
differs from the notation in the previous section) . From this we
obtain a further factorization of each $D_{u_{l}}^{\left(l\right)}$
by
\[
D_{u_{l}}^{\left(l\right)}=\prod_{1\le i\le k,i\neq l}\prod_{u_{i}}\gcd\left(D_{u_{1}}^{\left(1\right)},\ldots,D_{u_{k}}^{\left(k\right)}\right).
\]
 Define 
\[
D_{u_{1},\ldots,u_{k}}=\gcd\left(D_{u_{1}}^{\left(1\right)},\ldots.D_{u_{k}}^{\left(k\right)}\right).
\]
 Hence taking \ref{eq:expthistok-1} to the $k$th power we get 
\[
\frac{1}{p^{k}\cdot p^{k\omega\left(D\right)}}\sum\cdots\sum\prod_{\left(v_{1},\ldots,v_{k}\right)}\prod_{l\mid D_{v_{1},\ldots,v_{k}}}\chi_{l}\left(\prod_{\left(u_{1},\ldots,u_{k}\right)}D_{u_{1},\ldots,u_{k}}^{\Phi\left(u_{1},v_{1}\right)+\cdots+\Phi\left(u_{k},v_{k}\right)}\right)
\]
 where there are $k$ sums and each is over all factorizations of
$D$. To simplify notation we let $u=\left(u_{1},\ldots,u_{k}\right)$,
$v=\left(v_{1},\ldots,v_{k}\right)$, and $\Phi_{k}\left(u,v\right)=\sum_{i}\Phi\left(u_{i},v_{i}\right)$.
Then the expression becomes 
\[
\frac{1}{p^{k}\cdot p^{k\omega\left(D\right)}}\sum_{\left(D_{u}\right)}\prod_{\left(v\right)}\prod_{l\mid D_{v}}\chi_{l}\left(\prod_{\left(u\right)}D_{u}^{\Phi_{k}\left(u,v\right)}\right)
\]
where the sum is over $p^{2k}$-tuples of integers with 
\[
\prod_{u}D_{u}=D.
\]
 Now we sum over all characters to get 
\[
S_{k}\left(X\right)=\frac{1}{\left(p-1\right)\cdot p^{k}\cdot p^{k\omega\left(D\right)}}\sum_{\left(D_{u}\right)}\sum_{\left(\chi_{l}\right)}\prod_{\left(v\right)}\prod_{l\mid D_{v}}\chi_{l}\left(\prod_{\left(u\right)}D_{u}^{\Phi_{k}\left(u,v\right)}\right).
\]
Finally we sum over all $D<X$ such that $D^{p-1}$ is a discriminant
of a degree $p$ cyclic field which changes the above sum to be over
$p^{2k}$-tuples of integers whose prime factors are congruent to
$1$ mod $p$ or equal to $p$ and $\prod_{u}D_{u}<X$.

Thus we have proven 
\begin{prop}
\label{prop:mainexpression}For any $k$,
\[
\sum_{K,D_{K}<X^{p-1}}\left|\mathrm{im}\left(1-\sigma_{K}\right)\right|^{k}=\frac{1}{\left(p-1\right)\cdot p^{k}}\sum_{\left(D_{u}\right)}\sum_{\left(\chi_{l}\right)}\frac{\mu^{2}\left(\prod D_{u}\right)}{p^{k\omega\left(\prod D_{u}\right)}}\prod_{v}\prod_{l\mid D_{v}}\chi_{l}\left(\prod_{u}D_{u}^{\Phi_{k}\left(u,v\right)}\right)
\]
where the first sum on the right is over $p^{2k}$-tuples of integers
which are coprime and whose prime factors are congruent to $1$ mod
$p$ or equal to $p$ and $\prod_{u}D_{u}<X$ and the second sum is
over all tuples of order $p$ characters $\left(\chi_{p}\right)_{p\mid\prod D_{u}}$
and $u$ and $v$ run over $\left(\mathbb{Z}/p^{2}\mathbb{Z}\right)^{k}$. 
\end{prop}
Note we can alternately view the indices $u$ in $\left(\mathbb{F}_{p}\times\mathbb{F}_{p}\right)^{k}$.
We will adopt this notation in Section \ref{sec:Computing-the--th}.

The goal will now be to separate this expression into a main term
and an error term where the error term is $o\left(X\right)$. The
significance of $X$ is that it is asymptotically the number of degree
$p$ cyclic fields of discriminant bounded by $X^{p-1}$.

\section{Analytics Tools}

We list the analytic results that will be needed in the sequel. The
first two we take directly from \cite{fk1}.
\begin{lem}
\label{lem:There-exists-an}There exists an absolute constant $B_{0}$,
such that for every $X\ge3$ and every $l\ge0$ we have 
\[
\left|\left\{ n\le X\mid\omega\left(n\right)=l,\mu^{2}\left(n\right)=1\right\} \right|\le B_{0}\frac{X}{\log X}\frac{\left(\log\log X+B_{0}\right)^{l}}{l!}.
\]

\begin{lem}
\label{lem:Let--with}Let $\gamma\in\mathbb{R}$ with $\gamma>0$.
Then we have 
\end{lem}
\[
\sum_{X-Y<n<X}\gamma^{\omega\left(n\right)}\ll Y\left(\log X\right)^{\gamma-1}
\]
 for $2\le X\exp\left(-\sqrt{\log X}\right)\le Y\le X$.
\end{lem}
Let $\mathcal{O}=\mathbb{Z}\left[\zeta_{3}\right]$, the ring of integers
of the quadratic extension $\mathbb{Q}\left(\zeta_{3}\right)$. Let
$\left(\frac{x}{y}\right)_{3}$ denote the cubic residue symbol for
$x,y\in\mathcal{O}$ coprime. 

We will need the following results for estimating bilinear sums. They
are all versions of the large sieve inequality. The first two containing
the $\left(MN\right)^{\epsilon}$-type factor will be used when $M$
and $N$ are close together, and the latter two which do not contain
this factor will be used when $M$ and $N$ are far apart. The first
is Theorem 2 from \cite{heathbrowncubicguasssums}.
\begin{prop}
\label{prop:bs cubic epsilon }Let $c_{n}$ be a sequence of complex
numbers indexed by elements of $\mathcal{O}$. Then for any $\epsilon>0$
\[
\sum_{N\left(m\right)\le M}\left|\sum_{N\left(n\right)\le N}\mu^{2}\left(N\left(n\right)N\left(m\right)\right)c_{n}\left(\frac{n}{m}\right)_{3}\right|^{2}\ll_{\epsilon}\left(M+N+\left(MN\right)^{2/3}\right)\left(MN\right)^{\epsilon}\sum_{n}\left|c_{n}\right|^{2}
\]
where the sums are over elements of $\mathcal{O}$ congruent to 1
mod 3. 
\end{prop}
Next we have a version for cubic Dirichlet characters and sums over
integers, which is Theorem 1.4 from \cite{StephanBaierMatthewYoung}.
\begin{prop}
\label{prop:bs dirichlet epsilon}Let $c_{n}$ be a sequence of complex
numbers. Then for any $\epsilon>0$ 
\[
\sum_{Q<q<2Q}\sum_{\chi\mathrm{mod}q}\left|\sum_{M<m<2M}a_{m}\mu^{2}\left(m\right)\chi\left(m\right)\right|^{2}\ll_{\epsilon}\left(Q^{11/9}+Q^{2/3}M\right)\left(QM\right)^{\epsilon}\sum_{m}\left|a_{m}\right|^{2}
\]
where the $\chi$ are primitive Dirichlet characters satisfying $\chi^{3}=1$. 
\end{prop}
The next version is from \cite{IwaniecKowalski} and applies to all
Dirichlet characters.
\begin{prop}
\label{prop:bs dirichlet large}Let $c_{n}$ be a sequence of complex
numbers. Then for any $\epsilon>0$ 
\[
\sum_{q<Q}\sum_{\chi\mathrm{mod}q}\left|\sum_{M<m<2M}a_{m}\mu^{2}\left(m\right)\chi\left(m\right)\right|^{2}\ll_{\epsilon}\left(Q^{2}+M\right)\sum_{m}\left|a_{m}\right|^{2}
\]
 where the $\chi$ are primitive Dirichlet characters.
\end{prop}
The next is also from \cite{IwaniecKowalski}. In their terminology
a set of $\alpha_{r}=\left(\alpha_{r,1},\ldots,\alpha_{r,k}\right)\in\mathbb{R}^{k}$
is $\delta$-spaced if $\max_{i}\left|\alpha_{r,i}-\alpha_{r',i}\right|\ge\delta$
for all $r\ne r'$.
\begin{prop}
\label{prop:multivar large seive}Let $d\ge1$ and $\delta>0$ and
let $\alpha_{r}=\left(\alpha_{r,1},\ldots,\alpha_{r,d}\right)$ be
$\delta$-spaced points in $\mathbb{R}^{d}/\mathbb{Z}^{d}$ and $a_{n}$
a sequence in $\mathbb{C}$ indexed by $n=\left(n_{1},\ldots,n_{d}\right)\in\mathbb{Z}^{d}$
wih $1\le n_{i}\le N$. Then 
\[
\sum_{r}\sum_{n}\left|a_{n}\exp\left(2\pi i\left(n\cdot\alpha_{r}\right)\right)\right|^{2}\ll\left(\delta^{-d}+N^{d}\right)\left|a\right|.
\]
\end{prop}
Finally we will need a generalized version of Siegel-Walfisz for character
sums from \cite{GeneralizedSiegelWalfisz}. We state a slightly weaker
simplified version here.
\begin{prop}
\label{prop:general siegel walfisz}Let $A,\epsilon>0$. Let $K$
be Galois of degree $n$ and let $\chi$ be a finite Hecke character
of $K$ with conductor $f_{\chi}$. Then there exists a positive constant
$c=c(A,\epsilon)$, not depending on $K$ or $\chi$ such that 
\[
\sum_{N\left(p\right)\le x,\left(p,f_{\chi}\right)=1}\chi\left(p\right)=O\left(dx\log^{2}x\exp\left(-cn\left(\log x\right)^{1/2}/d\right)\right)
\]
 where $d=n^{3}\left|D_{K}N\left(f_{\chi}\right)\right|^{\epsilon}c^{-n}$.
The implied constant does not depend on $K$ or $\chi$.
\end{prop}

\section{\label{sec:Bounding-The-Error}Bounding The Error Term}

We start with the expression for $\sum_{K,D_{K}<D}\left|\mathrm{im}\left(1-\sigma_{K}\right)\right|^{k}$
which we derived in Proposition \ref{prop:mainexpression}, 
\begin{equation}
S_{k}\left(X\right)=\frac{1}{\left(p-1\right)\cdot p^{k}}\sum_{\left(D_{u}\right)}\sum_{\left(\chi_{l}\right)}\frac{\mu^{2}\left(\prod D_{u}\right)}{p^{k\omega\left(\prod D_{u}\right)}}\prod_{\left(v\right)}\prod_{p\mid D_{v}}\chi_{l}\left(\prod_{\left(u\right)}D_{u}^{\Phi_{k}\left(u,v\right)}\right)\label{eq:mainexpression}
\end{equation}
On the right we are summing over $p^{2k}$-tuples of integers with
$\prod D_{u}<X$ whose primes factors are congruent to $1$ mod $p$
or equal to $p$.

Fix $k\in\mathbb{Z}_{\ge1}$ and let $\Delta=1+\log^{-\left(p-1\right)\cdot p^{k}}X$.
Define $\mathbf{A}$ to be a tuple $\left(A_{i}\right)_{i=0}^{p^{2k}}$
of variables with each $A_{i}$ corresponding to $D_{i}$, and each
$A_{i}=\Delta^{j}$ for some $j\ge0$. We can partition $S_{k}\left(X\right)$
according to the various $\mathbf{A}$, by letting $S_{k}\left(X,\mathbf{A}\right)$
be the above sum but now restricted to tuples $\left(D_{i}\right)$
for which $A_{i}\le D_{i}\le\Delta A_{i}$ and $\prod_{i}D_{i}<X$.
Hence 
\[
S_{k}\left(X\right)=\sum_{\mathbf{A}}S_{k}\left(X,\mathbf{A}\right).
\]

Note that if $\Delta=1+\log^{-\left(p-1\right)\cdot p^{k}}X$ then
there are $O\left(\left(\log X\right)^{p^{2}\left(1+\left(p-1\right)\cdot p^{k}\right)}\right)$
possible $\mathbf{A}$ with $S_{k}\left(X,\mathbf{A}\right)$ not
empty. This is since there are $O\left(\left(\log X\right)^{\left(1+\left(p-1\right)\cdot p^{k}\right)}\right)$
choices for each $1<A_{i}\le X$. 

We now consider certain families of the $\mathbf{A}$ which make a
negligble contribution to the sum and hence can be removed. These
will be the same as the four families from \cite{fk1}, section 5.4.
In the case of the first two the argument is identical. For completeness
we reproduce their arguments here. The proofs involving the third
and fourth families require modification. 

First we reduce the sum to terms where all of the $D_{u}$ satisfy
$\omega\left(D_{u}\right)\le\Omega$, where we define $\Omega=ep^{2k}\left(\log\log X+B_{0}\right)$.
We restate the argument from \cite{fk1}.

Let $S_{0}$ be the terms in $\left(\ref{eq:mainexpression}\right)$where
not all of the $D_{u}$ sastisfy $\omega\left(D_{u}\right)\le\Omega$.
Let $n=\prod D_{u}$. Then
\begin{eqnarray*}
S_{0} & \ll & \sum_{n<X,\omega\left(n\right)>\Omega}\mu^{2}\left(n\right)p^{k\omega\left(n\right)}.
\end{eqnarray*}
 Then splitting the sum up by number of prime factors and applying
Lemma \ref{lem:There-exists-an} we get the bound
\begin{eqnarray*}
\sum_{n<X,\omega\left(n\right)>\Omega}\mu^{2}\left(n\right)p^{k\omega\left(n\right)} & \ll & \sum_{l\ge\Omega}\frac{X}{\log X}p^{kl}\frac{\left(\log\log X+B_{0}\right)^{l}}{l!}\\
 & \ll & \frac{X}{\log X}\sum_{l\ge\Omega}\left(\frac{p^{k}\left(\log\log X+B_{0}\right)}{l/e}\right)^{l}\\
 & \ll & \frac{X}{\log X}\sum_{l\ge\Omega}\left(\frac{1}{p^{k}}\right)^{l}\\
 & \ll & \frac{X}{\log X}
\end{eqnarray*}
 where in the last inequality we are using $l/e\ge p^{2k}\left(\log\log X+B_{0}\right)$.
Thus we can assume in the remainder that all variables $D_{u}$ satisfy
$\omega\left(D_{u}\right)\le\Omega$. We will only need this fact
to bound family 4.

\subsection{The first family}

Note that it is possible to have an $\mathbf{A}$ for which $S_{k}\left(X,\mathbf{A}\right)$
is not empty, but

\begin{equation}
\prod_{i}\Delta A_{i}>X\label{eq:family1}
\end{equation}
making the restriction $D<X$ necessary. By Lemma \ref{lem:Let--with}
we have
\begin{eqnarray*}
\sum_{\mathbf{A}\mathrm{\ satisifes\ }\left(\ref{eq:family1}\right)}S_{k}\left(X,\mathbf{A}\right) & \ll & \sum_{\Delta^{-p^{2k}}X\le D\le X}p^{2k\omega\left(D\right)}\left(p-1\right)^{\omega\left(D\right)}\left(\frac{1}{p}\right)^{k\omega\left(D\right)}\\
 & = & \sum_{\Delta^{-p^{2k}}X\le D\le X}\left(\left(p-1\right)\cdot p^{k}\right)^{\omega\left(D\right)}\\
 & \ll & \left(1-\Delta^{-p^{2k}}\right)X\left(\log X\right)^{\left(p-1\right)\cdot p^{k}-1}\\
\end{eqnarray*}
 Using that $\left(1+x\right)^{\alpha}=1+\alpha x+O\left(x^{2}\right)$
for $x\longrightarrow0$, we get $\Delta^{-p^{2k}}=\left(1+\log^{-\left(p-1\right)\cdot p^{k}}X\right)^{-p^{2k}}=1-p^{2k}\log^{-\left(p-1\right)\cdot p^{k}}X+O\left(\log^{-2\left(\left(p-1\right)\cdot p^{k}\right)}X\right)$.
This gives the bound
\begin{eqnarray*}
 & \ll & \left(p^{2k}\log^{-\left(p-1\right)\cdot p^{k}}X+O\left(\log^{-2\left(\left(p-1\right)\cdot p^{k}\right)}X\right)\right)X\left(\log X\right)^{\left(p-1\right)\cdot p^{k}-1}\\
 & \ll & X/\log X.
\end{eqnarray*}

\subsection{The second family}

In the remaining sections we will need the following quantities. Let
\[
X^{\dagger}=\log^{4\left(1+p^{2k}\left(1+\left(p-1\right)\cdot p^{k}\right)\right)}X
\]
\[
X^{\ddagger}=\exp\left(\log^{\eta}X\right)
\]
for some small $\eta$. Consider the $\mathbf{A}$ which satisfy 
\begin{equation}
\mbox{at most \ensuremath{p^{k-1}} variables }A_{i}>X^{\ddagger}.\label{eq:family2}
\end{equation}
Let $r$ be the number of variables greater than $X^{\ddagger}$.
We have

\begin{eqnarray*}
\sum_{\mathbf{A}\mathrm{\ satisifes\ }\left(\ref{eq:family2}\right)}S_{k}\left(X,\mathbf{A}\right) & \ll & \sum_{r=0}^{p^{k-1}}\sum_{m<\left(X^{\ddagger}\right)^{p^{2k}-r}}\mu^{2}\left(m\right)\tau_{p^{2k}-r}\left(m\right)\left(\frac{p-1}{p^{k}}\right)^{\omega\left(m\right)}\\
 &  & \times\sum_{n<X/m}\mu^{2}\left(n\right)\tau_{r}\left(n\right)\left(\frac{p-1}{p^{k}}\right)^{\omega\left(n\right)}\\
 & = & \sum_{r=0}^{p^{k-1}}\sum_{m<\left(X^{\ddagger}\right)^{p^{2k}-r}}\mu^{2}\left(m\right)\left(p^{2k}-r\right)^{\omega\left(m\right)}\left(\frac{p-1}{p^{k}}\right)^{\omega\left(m\right)}\\
 &  & \times\sum_{n<X/m}\mu^{2}\left(n\right)\left(\frac{\left(p-1\right)r}{p^{k}}\right)^{\omega\left(n\right)}
\end{eqnarray*}
 Then we get using Lemma \ref{lem:Let--with} 

\begin{eqnarray*}
 & \ll & \sum_{r=0}^{p^{k-1}}\sum_{m<\left(X^{\ddagger}\right)^{p^{k}-r}}\frac{\left(\left(p-1\right)\cdot p^{k}\right)^{\omega\left(m\right)}}{m}\left(X/m\right)\left(\log X\right)^{\left(p-1\right)r/p^{k}-1}\\
 & \ll & X\left(\log X\right)^{\eta\left(p-1\right)\cdot p^{k}-1/p}.
\end{eqnarray*}

\subsection{The third family, the case $p=3$}

For the third and fourth families we will assume $p=3$ and bound
the error term unconditionally. Afterwards we will handle the case
of general $p$ case under the assumption of GRH.

Before continuing we use cubic reciprocity to rewrite the expression
we have been working with so far. Let $\zeta$ be a cube root of unity.
We recall some facts about the field $\mathbb{Q}\left(\zeta\right)$
and the cubic residue symbol which can be found in \cite{StephanBaierMatthewYoung}.
The ring of integers of $\mathbb{Q}\left(\zeta_{3}\right)$ is $\mathcal{O}=\mathbb{Z}\left[\zeta\right]$.
It is a principle ideal domain and every ideal $\left(n\right)\subset\mathcal{O}$
with $\left(n,3\right)=1$ has a unique generator $n$ which satisfies
$n\equiv1\mod3$. The only prime which ramifies is $3=\left(1-\zeta\right)^{2}$.
The primes of $\mathbb{Z}$ which split in $\mathcal{O}$ are exactly
the ones congruent to $1\mod3$ in $\mathbb{Z}.$ We can choose a
basis $\left\{ 1,\zeta\right\} $ for $\mathcal{O}$, so that every
element can be written as $a+b\zeta$, and then $N\left(a+b\zeta\right)=a^{2}+b^{2}-ab$.
Using this it can be shown that $\left|\left\{ a+b\zeta\in\mathcal{O}\mid N\left(a+b\zeta\right)\le A\right\} \right|\ll A$
(both $a$ and \textbf{$b$ }are $O\left(A^{1/2}\right)$).

In the remainder of the paper all summations over elements of $\mathcal{O}$
will be restricted to those which are products of split primes and
which are congruent to $1\mod3$ (by the above these can be viewed
as summations over integral ideals of $\mathcal{O}$). 

Denote by $\left(\frac{n}{m}\right)_{3}$ the cubic residue symbol,
defined for $\left(m\right)\neq\left(1-\zeta\right)$ and $n$ coprime
to $m$. For $l\neq3$ the characters $\chi_{l}$ are cubic dirichlet
characters and hence for $n\in\mathbb{Z}$ we have $\chi_{l}\left(n\right)=\left(\frac{n}{\pi}\right)_{3}$
for some $\pi$ with $N\left(\pi\right)=l$. 

We define some terminology which will be used in the remainder of
the paper. We define indices $u$ and $v$ to be linked if $\Phi_{k}\left(u,v\right)\neq0$.
Otherwise we say they are unlinked. 

Consider the third family which consists of those $\mathbf{A}$ such
that there are two linked indices $A_{u}$ and $A_{v}$, and 
\begin{equation}
A_{u},A_{v}>X^{\dagger}.\label{eq:family3}
\end{equation}

Fix such an $\mathbf{A}$ and two linked indices $u,v$. We consider
two cases: case 1 will be when both $\Phi_{k}\left(u,v\right)$ and
$\Phi_{k}\left(v,u\right)$ are nonzero in \ref{eq:mainexpression}
and case 2 will be when only one of these is nonzero.

\textit{Case 1:} Both $\Phi_{k}\left(u,v\right)$ and $\Phi_{k}\left(v,u\right)$
are nonzero.
\begin{lem}
\label{lem:For-any-linked}For any linked indices $u$ and $v$ with
$\Phi_{k}\left(u,v\right)$ and $\Phi_{k}\left(v,u\right)$ both nonzero,
\begin{equation}
S_{k}\left(X,\mathbf{A}\right)\ll\sum_{D_{w},w\neq u,v}\left|\sum_{d_{v},d_{u}\in\mathcal{O}}a\left(d_{u}\right)a\left(d_{v}\right)\left(\frac{d_{u}}{d_{v}}\right)_{3}\right|\label{eq:residuesymbol}
\end{equation}
with $\left|a\left(d_{i}\right)\right|\le1$, where the summations
over elements of $d_{i}\in\mathcal{O}$ congruent to $1\mod3$ which
are products of split primes, $\mu^{2}\left(N\left(d_{i}\right)\right)=1$
and $N\left(d_{i}\right)\le\Delta A_{i}$.
\end{lem}
\begin{proof}
Fix two indices $u',v'$ such that $\Phi_{k}\left(u,v\right)\neq0$
and $\Phi_{k}\left(u,v\right)\neq0$. In the following we write $D=\prod_{w}D_{w}$
and $D'=D_{u}D_{v}$. Then from $\left(\ref{eq:mainexpression}\right)$
we get

\begin{align*}
S_{k}\left(X,\mathbf{A}\right) & =\frac{1}{2\cdot3^{k}}\sum_{\left(D_{w}\right)}\sum_{\left(\chi_{l}\right)}\frac{\mu^{2}\left(D\right)}{3^{k\omega\left(D\right)}}\prod_{y}\prod_{l\mid D_{y}}\chi_{l}\left(\prod_{z}D_{z}^{\Phi_{k}\left(z,y\right)}\right)\\
 & =\frac{1}{2\cdot3^{k}}\sum_{D_{w},w\neq u,v}\sum_{\left(\chi_{l}\right),l\mid D/D'}\frac{1}{3^{k\omega\left(D/D'\right)}}\sum_{D_{u},D_{v}}\sum_{\left(\chi_{l}\right),l\mid D'}\frac{\mu^{2}\left(D\right)}{3^{k\omega\left(D'\right)}}\\
 & \times\prod_{y}\prod_{l\mid D_{y}}\chi_{l}\left(\prod_{z}D_{z}^{\Phi_{k}\left(z,y\right)}\right)
\end{align*}
We can split this sum into 3 pieces according to whether $3\mid D_{u}$,
$3\mid D_{v}$ or $3\mid D/D'$ and apply the next argument to each
case seperately. We illustrate the case $3\mid D_{v}$, the others
being handled similarly. 

In the inner sum over $D_{u},D_{v}$ we can replace the $\chi_{l}$
with cubic residue symbols:

\begin{align}
 & \sum_{D_{u},D_{v}}\sum_{\left(\chi_{l}\right),l\mid D'}\frac{\mu^{2}\left(D\right)}{3^{k\omega\left(D'\right)}}\prod_{y}\prod_{l\mid D_{y}}\chi_{l}\left(\prod_{z}D_{z}^{\Phi_{k}\left(z,y\right)}\right)\nonumber \\
 & \ll\sum_{d_{v},d_{u}\in\mathcal{O}}b\left(d_{u}\right)b\left(d_{v}\right)\left(\frac{D_{u}}{d_{v}}\right)_{3}^{\Phi_{k}\left(u,v\right)}\left(\frac{D_{v}/3}{d_{u}}\right)_{3}^{\Phi_{k}\left(v,u\right)}\nonumber \\
 & =\sum_{d_{v},d_{u}\in\mathcal{O}}a\left(d_{u}\right)a\left(d_{v}\right)\left(\frac{D_{u}}{d_{v}}\right)_{3}\left(\frac{D_{v}/3}{d_{u}}\right)_{3}\label{eq:expanding}
\end{align}
where in the second line $D_{u},D_{v}$ now denote $N\left(d_{u}\right),N\left(d_{v}\right)$
and 
\[
b\left(d_{u}\right)=\frac{\mu^{2}\left(D\right)}{3^{k\omega\left(D_{u}\right)}}\prod_{y\neq u,v}\prod_{l\mid D_{y}}\chi_{l}\left(D_{u}^{\Phi_{k}\left(u,y\right)}\right)\left(\frac{3^{\Phi_{k}\left(u,v\right)}\prod_{y\neq u,v}D_{v}^{\Phi_{k}\left(y,u\right)}}{d_{u}}\right)_{3}
\]
and similarly for $b\left(d_{v}\right)$ (which will contain the factor
$\chi_{3}\left(D_{u}\right)$). The elements $d_{v},d_{u}$ in the
summation satisfy $A_{v}\le N\left(d_{v}\right)\le\Delta A_{v}$ and
$A_{u}\le N\left(d_{u}\right)\le\Delta A_{u}$ and $d_{u}\equiv1\mod3$
and $d_{v}/\left(1-\zeta\right)\equiv1\mod3$. We can instead sum
over $d_{v}$ such that $A_{v}/3\le N\left(d_{v}\right)\le\Delta A_{v}/3$
and $d_{v}\equiv1\mod3$. Note $\Phi_{k}$ is either 1 or 2, and squaring
a cubic character is the same as conjugating it. Since we are summing
over a set of characters closed under conjugation removing $\Phi_{k}$
from the exponent permutes the coefficients.

Furthermore by cubic reciprocity, for any $D_{u}=N\left(d_{u}\right)$
and $D_{v}=N\left(d_{v}\right)$ not divisible by 3 we have
\begin{eqnarray}
\left(\frac{D_{u}}{d_{v}}\right)_{3}\left(\frac{D_{v}}{d_{u}}\right)_{3} & = & \left(\frac{d_{u}}{d_{v}}\right)_{3}\left(\frac{\overline{d_{u}}}{d_{v}}\right)_{3}\left(\frac{d_{v}}{d_{u}}\right)_{3}\left(\frac{\overline{d_{v}}}{d_{u}}\right)_{3}\nonumber \\
 & = & \left(\frac{d_{u}}{d_{v}}\right)_{3}^{2}\left(\frac{\overline{d_{u}}}{d_{v}}\right)_{3}\left(\frac{\overline{d_{u}}}{d_{v}}\right)_{3}^{2}\label{eq:cubic recip}\\
 & = & \left(\frac{\overline{d_{u}}}{\overline{d_{v}}}\right)_{3}.\nonumber 
\end{eqnarray}
This proves the lemma.
\end{proof}
We will follow the standard strategy of bounding bilinear sums using
Cauchy-Schwarz followed by a large seive type bound. By Cauchy-Schwarz
we have
\begin{eqnarray*}
\left|\sum_{d_{v}\in\mathcal{O}}\sum_{d_{u}\in\mathcal{O}}a\left(d_{u}\right)a\left(d_{v}\right)\left(\frac{d_{u}}{d_{v}}\right)_{3}\right| & \ll & A_{v}^{1/2}\left(\sum_{d_{v}\in\mathcal{O}}\left|\sum_{d_{u}\in\mathcal{O}}a\left(d_{u}\right)\left(\frac{d_{u}}{d_{v}}\right)_{3}\right|^{2}\right)^{1/2}
\end{eqnarray*}
 Now we focus on bounding the summation on the right hand side. For
$r,n\in\mathcal{O}$, with $\left(r,n\right)=1$, $n\equiv1\left(3\right)$
and $\mu^{2}\left(N\left(n\right)\right)=1$ define 
\[
g\left(r,n\right)=\sum_{d\left(\mathrm{mod}n\right)}\left(\frac{d}{n}\right)_{3}\check{e}\left(rd/n\right)
\]
 to be the generalized cubic gauss sum where $\check{e}\left(z\right)=\exp\left(2\pi i\left(z+\overline{z}\right)\right)$.
The following facts can be found in \cite{StephanBaierMatthewYoung}.
It satisfies the property
\[
g\left(rs,n\right)=\overline{\left(\frac{s}{n}\right)_{3}}g\left(r,n\right).
\]
We will write $g\left(n\right)=g\left(1,n\right)$.

Thus the above sum becomes
\begin{eqnarray*}
\sum_{d_{v}}\left|\sum_{d_{u}}a\left(d_{u}\right)\left(\frac{d_{u}}{d_{v}}\right)_{3}\right|^{2} & \le & \sum_{d_{v}}\frac{1}{\left|g\left(d_{v}\right)\right|^{2}}\left|\sum_{d_{u}}a\left(d_{u}\right)g\left(d_{u},d_{v}\right)\right|^{2}\\
 & \le & \sum_{d_{v}}\frac{1}{\left|g\left(d_{v}\right)\right|^{2}}\left|\sum_{d\left(\mathrm{mod}d_{v}\right)}\left(\frac{d}{d_{v}}\right)_{3}\sum_{d_{u}}a\left(d_{u}\right)\check{e}\left(dd_{u}/d_{v}\right)\right|^{2}\\
 & \le & \sum_{d_{v}}\frac{1}{\left|g\left(d_{v}\right)\right|^{2}}\sideset{}{_{\chi}^{*}}\sum\left|\sum_{d\left(\mathrm{mod}d_{v}\right)}\chi\left(d\right)\sum_{d_{u}}a\left(d_{u}\right)\check{e}\left(dd_{u}/d_{v}\right)\right|^{2}\\
 & \le & \sum_{d_{v}}\sum_{d\left(\mathrm{mod}d_{v}\right)}\left|\sum_{d_{u}}a\left(d_{u}\right)\check{e}\left(dd_{u}/d_{v}\right)\right|^{2}
\end{eqnarray*}
 where the summation $\sideset{}{_{\chi}^{*}}\sum$ is over primitive
characters of $\left(\mathcal{O}/d_{v}\right)^{\times}$ and in the
last line we are opening the square and using orthogonality of characters.
Using our previously chosen basis $\left\{ 1,\zeta\right\} $ we can
rewrite a summation over elements of $\mathcal{O}$ as one over $\mathbb{Z}^{2}$.We
obtain the bound
\begin{eqnarray*}
\sum_{d_{v}}\sum_{d\left(\mathrm{mod}d_{v}\right)}\left|\sum_{d_{u}}a\left(d_{u}\right)\check{e}\left(dd_{u}/d_{v}\right)\right|^{2} & \ll & \sum_{d_{v}}\sum_{d\left(\mathrm{mod}d_{v}\right)}\left|\sum_{s_{1},s_{2}}a\left(s_{1},s_{2}\right)\check{e}\left(dd_{u}/d_{v}\right)\right|^{2}
\end{eqnarray*}
 where $s_{i}\in\mathbb{Z}$ and $s_{i}\ll A_{u}^{1/2}$. 

Write $d/d_{v}=\left(d_{1},d_{2}\right)$ and $d_{u}=\left(s_{1},s_{2}\right)$.
Then $dd_{u}/d_{v}=\left(d_{1}s_{1}-d_{2}s_{2},d_{1}s_{2}+d_{2}s_{1}-d_{2}s_{2}\right)$
and 
\begin{eqnarray*}
\mathrm{tr}\left(dd_{u}/d_{v}\right) & = & s_{1}\left(2d_{1}-d_{2}\right)+s_{2}\left(-d_{1}-d_{2}\right)\\
 & = & \left(s_{1},s_{2}\right)\cdot\left(2d_{1}-d_{2},d_{1}-d_{2}\right).
\end{eqnarray*}
 Hence we can rewrite the right hand side of the above inequality
as
\[
\sum_{d_{v}}\sum_{d\left(\mathrm{mod}d_{v}\right)}\left|\sum_{s_{1},s_{2}}a\left(s_{1},s_{2}\right)\exp\left(2\pi i\left(s_{1},s_{2}\right)\cdot\left(2d_{1}-d_{2},d_{1}-d_{2}\right)\right)\right|^{2}
\]
 We want to apply a multivariable large seive inequality of Proposition
\ref{prop:multivar large seive}, so we will first show that the set
$S$ of $\left(2d_{1}-d_{2},d_{1}-d_{2}\right)$ (where $d/d_{v}$
runs over the above summation) is $1/A_{v}$-spaced. Note that for
any $a/c,b/c\in\mathbb{Q}$ we have $\left|a/c-b/c\right|\ge1/c$.
Hence the spacing of a set is determined by the denominators of the
coordinates of its elements. 

First note that the set of $\left(d_{1},d_{2}\right)$ obtained from
$d/d_{v}$ is distinct in $\left(\mathbb{Q}/\mathbb{Z}\right)^{2}$.
If not this implies $d/d_{v}=d'/d_{v}'+n$ for some $n\in\mathcal{O}$,
so $dd_{v}'=d'd_{v}+nd_{v}d_{v}'$ and noting that $\left(d,d_{v}\right)=\left(d',d_{v}'\right)=1$
and considering divisors of both sides gives a contradiction (recall
$\mathcal{O}$ is a UFD). Furthermore $\left(d_{1},d_{2}\right)\longmapsto\left(2d_{1}-d_{2},d_{1}-d_{2}\right)$
is a linear map which is invertible over $\mathbb{Z}$, hence the
elements of $S$ are all distinct. We can write
\[
\frac{d}{d_{v}}=\frac{d\overline{d_{v}}}{N\left(d_{v}\right)}=\frac{a}{N\left(d_{v}\right)}+\frac{b}{N\left(d_{v}\right)}\zeta
\]
 for some $a,b\in\mathbb{Z}$. Since $N\left(d_{v}\right)\le A_{v}$
it follows that $S$ is $1/A_{v}$-spaced as required. Thus by Proposition
\ref{prop:multivar large seive} we get 
\begin{eqnarray*}
\sum_{d_{v}}\sum_{d\left(\mathrm{mod}d_{v}\right)}\left|\sum_{s_{1},s_{2}}a\left(s_{1},s_{2}\right)\exp\left(2\pi i\left(s_{1},s_{2}\right)\cdot\left(2d_{1}-d_{2},d_{1}-d_{2}\right)\right)\right|^{2} & \ll & \left(A_{v}^{2}+A_{u}\right)A_{u}.
\end{eqnarray*}
Returning to the expression we started with we get
\begin{eqnarray}
\sum_{d_{v}\in\mathcal{O}}\left|\sum_{d_{u}\in\mathcal{O}}a\left(d_{u}\right)a\left(d_{v}\right)\left(\frac{d_{u}}{d_{v}}\right)_{3}\right| & \ll & A_{v}^{1/2}\left(\left(A_{v}^{2}+A_{u}\right)A_{u}\right)^{1/2}\nonumber \\
 & = & A_{v}A_{u}\left(\frac{A_{v}}{A_{u}}+\frac{1}{A_{v}}\right)^{1/2}.\label{eq:finalboundfamily3case1}
\end{eqnarray}
By symmetry we can also bound this by $A_{v}A_{u}\left(\left(\frac{A_{u}}{A_{v}}+\frac{1}{A_{u}}\right)\right){}^{1/2}$.

By symmetry we can let $A_{v}\le A_{u}$. First suppose $A_{v}^{2}<A_{u}$.
By the above bound $\left(\ref{eq:finalboundfamily3case1}\right)$
we get 
\begin{eqnarray*}
S_{k}\left(X,\mathbf{A}\right) & \ll & \sum_{D_{w},w\neq u,v}A_{v}A_{u}\left(\frac{A_{v}}{A_{u}}+\frac{1}{A_{v}}\right)^{1/2}\\
 & \ll & X\left(\frac{1}{A_{u}^{1/2}}+\frac{1}{A_{v}}\right)^{1/2}\\
 & \ll & X/\log^{1+9^{k}\left(1+2\cdot3^{k}\right)}X.
\end{eqnarray*}
Now suppose $A_{u}\le A_{v}^{2}$. Then by Proposition \ref{prop:bs cubic epsilon }
we get the bound, for any $\epsilon>0$, 
\begin{eqnarray*}
S_{k}\left(X,\mathbf{A}\right) & \ll & \sum_{D_{w},w\neq u,v}A_{v}^{1/2}\left(\left(A_{u}+A_{v}+\left(A_{u}A_{v}\right)^{2/3}\right)\left(A_{u}A_{v}\right)^{\epsilon}A_{u}\right)^{1/2}\\
 & \ll & X\left(\left(\frac{1}{A_{v}}+\frac{1}{A_{u}}+\frac{1}{\left(A_{u}A_{v}\right)^{1/3}}\right)\left(A_{u}A_{v}\right)^{\epsilon}\right)^{1/2}\\
 & \ll & X\left(\frac{1}{X^{\dagger^{1/2}}}\right)^{1/2}\\
 & \ll & X/\log^{1+9^{k}\left(1+2\cdot3^{k}\right)}X.
\end{eqnarray*}

\textit{Case 2:} Only one of $\Phi_{k}\left(u,v\right)$ and $\Phi_{k}\left(v,u\right)$
is nonzero.

Without loss of generality assume $\Phi_{k}\left(u,v\right)$ is nonzero.
\begin{lem}
For any linked indices $u$ and $v$ with $\Phi_{k}\left(u,v\right)$
and $\Phi_{k}\left(v,u\right)$ both nonzero, 
\begin{equation}
S_{k}\left(X,\mathbf{A}\right)\le\sum_{D_{w},w\neq u,v}\left|\sum_{d_{v}\in\mathcal{O}}\sum_{D_{u}\in\mathbb{Z}}a\left(D_{u}\right)a\left(d_{v}\right)\left(\frac{D_{u}}{d_{v}}\right)_{3}\right|.\label{eq:dirichletchar}
\end{equation}
with $\left|a\left(d_{i}\right)\right|\le1$, where the summations
over elements of $d_{i}\in\mathcal{O}$ congruent to $1\mod3$ which
are products of split primes, $\mu^{2}\left(N\left(d_{i}\right)\right)=1$
and $N\left(d_{i}\right)\le\Delta A_{i}$ and $D_{i}\in\mathbb{Z}$
a product of primes congruent to $1\mod3$, $\mu^{2}\left(D_{i}\right)=1$
and $D_{i}\le\Delta A_{i}$.
\end{lem}
\begin{proof}
This is a less involved version of the proof of Lemma \ref{lem:For-any-linked}.
\end{proof}
The above expression is no longer symmetric in $u$ and $v$ hence
we must consider several subcases. In the case when the variables
$A_{u}$ and $A_{v}$ are close together, specifically $A_{u}<A_{v}<A_{u}^{2}$
or $A_{v}<A_{u}<A_{v}^{2}$ we apply Theorem \ref{prop:bs dirichlet epsilon}.
For example if $A_{u}<A_{v}<A_{u}^{2}$ we get 
\begin{eqnarray*}
S_{k}\left(X,\mathbf{A}\right) & \le & \sum_{D_{w},w\neq u,v}A_{v}^{1/2}\left(\sum_{d_{v}}\left|\sum_{D_{u}}a\left(D_{u}\right)\left(\frac{D_{u}}{d_{v}}\right)_{3}\right|^{2}\right)^{1/2}\\
 & \ll & \sum_{D_{w},w\neq u,v}A_{v}^{1/2}\left(\left(A_{v}^{11/9}+A_{v}^{2/3}A_{u}\right)\left(A_{u}A_{v}\right)^{\epsilon}A_{u}\right)^{1/2}\\
 & \ll & X\left(\left(\frac{A_{v}^{2/9}}{A_{u}}+\frac{1}{A_{v}^{1/3}}\right)\left(A_{u}A_{v}\right)^{\epsilon}\right)^{1/2}\\
 & \ll & X/\log^{\left(1+9^{k}\left(1+2\cdot3^{k}\right)\right)/2}X
\end{eqnarray*}

Now suppose $A_{u}^{2}<A_{v}$. Define a cubic Hecke character
\[
\psi:\left(d_{v}\right)\longmapsto\left(\frac{D_{u}}{d_{v}}\right)_{3}
\]
 of modulus $9D_{u}$.

As in case 1 define the Gauss sum
\[
g\left(r,n\right)=\sum_{d\in\left(\mathcal{O}/9D_{u}\right)^{\times}}\psi\left(d\right)\check{e}\left(rd/n\right)
\]
which satisfies for $\left(r,9D_{u}\right)=1$
\begin{eqnarray*}
g\left(r,D_{u}\right) & = & \overline{\psi\left(r\right)}g\left(1,D_{u}\right).
\end{eqnarray*}
As in case 1 we get 
\begin{eqnarray*}
\sum_{D_{u}}\left|\sum_{d_{v}}a\left(d_{v}\right)\left(\frac{D_{u}}{d_{v}}\right)_{3}\right|^{2} & \ll & \sum_{D_{u}}\sum_{d\in\left(\mathcal{O}/9D_{u}\right)^{\times}}\left|\sum_{s_{1},s_{2}}a\left(s_{1},s_{2}\right)\check{e}\left(dd_{v}/D_{u}\right)\right|^{2}
\end{eqnarray*}
where $s_{i}\in\mathbb{Z}$ and $s_{i}\ll A_{v}^{1/2}$. 

Write $d/D_{u}=\left(d_{1},d_{2}\right)$ and $d_{v}=\left(s_{1},s_{2}\right)$.
Then $dd_{v}/D_{u}=\left(d_{1}s_{1}-d_{2}s_{2},d_{1}s_{2}+d_{2}s_{1}-d_{2}s_{2}\right)$
and 
\begin{eqnarray*}
\mathrm{tr}\left(dd_{v}/D_{u}\right) & = & s_{1}\left(2d_{1}-d_{2}\right)+s_{2}\left(-d_{1}-d_{2}\right)\\
 & = & \left(s_{1},s_{2}\right)\cdot\left(2d_{1}-d_{2},d_{1}-d_{2}\right).
\end{eqnarray*}
 Hence we can rewrite the right hand side of the above inequality
as
\[
\sum_{D_{u}}\sum_{d\in\left(\mathcal{O}/9D_{u}\right)^{\times}}\left|\sum_{s_{1},s_{2}}a\left(s_{1},s_{2}\right)\exp\left(2\pi i\left(s_{1},s_{2}\right)\cdot\left(2d_{1}-d_{2},d_{1}-d_{2}\right)\right)\right|^{2}
\]
By the same argument as before the set $S$ of $\left(2d_{1}-d_{2},d_{1}-d_{2}\right)$
(where $d/D_{u}$ runs over the above summation) is distinct and $1/A_{u}$-spaced
since we can write 
\[
\frac{d}{D_{u}}=\frac{a}{D_{u}}+\frac{b}{D_{u}}\zeta
\]
 for some $a,b\in\mathbb{Z}$. Thus by Proposition \ref{prop:multivar large seive}
we get 
\begin{eqnarray*}
\sum_{D_{u}}\sum_{d\in\left(\mathcal{O}/9D_{u}\right)^{\times}}\left|\sum_{s_{1},s_{2}}a\left(s_{1},s_{2}\right)\exp\left(2\pi i\left(s_{1},s_{2}\right)\cdot\left(2d_{1}-d_{2},d_{1}-d_{2}\right)\right)\right|^{2} & \ll & \left(A_{u}^{2}+A_{v}\right)A_{v}.
\end{eqnarray*}
Returning to the expression we started with we get
\begin{eqnarray*}
\left|\sum_{d_{v}\in\mathcal{O}}\sum_{D_{u}\in\mathbb{Z}}a\left(D_{u}\right)a\left(d_{v}\right)\left(\frac{D_{u}}{d_{v}}\right)_{3}\right| & \ll & A_{u}^{1/2}\left(\left(A_{u}^{2}+A_{v}\right)A_{v}\right)^{1/2}\\
 & = & A_{u}A_{v}\left(\frac{A_{u}}{A_{v}}+\frac{1}{A_{u}}\right)^{1/2}.
\end{eqnarray*}
By the above bound we get 
\begin{eqnarray*}
S_{k}\left(X,\mathbf{A}\right) & \ll & \sum_{D_{w},w\neq u,v}A_{u}A_{v}\left(\frac{A_{u}}{A_{v}}+\frac{1}{A_{u}}\right)^{1/2}\\
 & \ll & X\left(\frac{1}{A_{v}^{1/2}}+\frac{1}{A_{u}}\right)^{1/2}\\
 & \ll & X/\log^{1+9^{k}\left(1+2\cdot3^{k}\right)}X.
\end{eqnarray*}
 The case when $A_{v}^{2}<A_{u}$ follows directly from \ref{prop:bs dirichlet large}.

Finally summing over all such $\mathbf{A}$ we get 
\[
\sum_{\mathbf{A}\mathrm{\ satisifes\ }\left(\ref{eq:family3}\right)}S_{k}\left(X,\mathbf{A}\right)\ll X/\log X.
\]

\subsection{The fourth family, the case $p=3$}

Now consider the fourth family which consists of those $\mathbf{A}$
which are not in the third family and such that there are two linked
indices $A_{u}$ and $A_{v}$ and 
\begin{equation}
A_{u}>X^{\ddagger},2\le A_{v}<X^{\dagger}.\label{eq:family4}
\end{equation}
Note the condition $2\le A_{v}<X^{\dagger}$ is forced by the assumption
that $\mathbf{A}$ is not in the third family. We in fact consider
the collection of all indices $v$ which satisfy the above condition.
Denote this set by $S$.
\begin{lem}
For $\mathbf{A}$, $u$ and $S$ as defined above we have 
\[
S_{k}\left(X,\mathbf{A}\right)\ll\sum_{D_{v},v\neq u}\sum_{\left(\chi_{l}\right),l\mid D_{v}}\frac{1}{3^{\omega\left(\prod_{w\neq u}D_{w}\right)}}\left|\sum_{d_{u}\in\mathcal{O}}\frac{\mu^{2}\left(D\right)}{3^{\omega\left(D_{u}\right)}}\psi\left(d_{u}\right)\right|
\]
where we denote $D_{u}=N\left(d_{u}\right)$ and $D=\prod_{w}D_{w}$
and the summation is over $d_{u}\equiv1\mod3$ which are products
of split primes and $A_{u}\le N\left(d_{u}\right)\le\Delta A_{u}$
and $D_{i}\in\mathbb{Z}$ a product of primes congruent to $1\mod3$
and $A_{v}\le D_{v}\le\Delta A_{v}$. Furthermore $\psi$ is a Hecke
character given by
\end{lem}
\[
\psi\left(d_{u}\right)=\left(\frac{\prod_{v}D_{v}^{\Phi_{k}\left(v,u\right)}}{d_{u}}\right)_{3}\prod_{v\in S}\prod_{l\mid D_{v}}\chi_{l}\left(D_{u}^{\Phi_{k}\left(u,y\right)}\right)\chi_{3}\left(D_{u}\right)^{i}.
\]

\begin{proof}
The proof is similar to that of Lemma \ref{lem:For-any-linked}. As
in the previous subsection we consider separately the cases when $3\mid D_{u}$
and $3\nmid D_{u}$. In the first case we replace $d_{u}$ by $d_{u}/\left(1-\zeta\right)$
and sum in the range $A_{u}/3\le N\left(d_{u}\right)\le\Delta A_{u}/3$.
The factor of $\chi_{3}$ in the definition of $\psi$ appears only
when $3\mid D_{v}$ for some $v\in S$.
\end{proof}
We remark that the $v\in S$ by assumption satisfy $A_{v}<X^{\dagger}$.
The modulus of $\psi$ is $f_{\psi}=9\prod_{v\in S}D_{v}$ and $N\left(f_{\psi}\right)\le81\left(X^{\dagger}\right)^{2\cdot9^{k}}$.
Note the $\mu^{2}$ factor means the non-zero terms of the summation
above come from $d_{u}$ whose prime factors are all split. 

Now looking at the inner sum

\begin{eqnarray*}
\sum_{d_{u}}\frac{\mu^{2}\left(D\right)}{3^{\omega\left(d_{u}\right)}}\psi\left(d_{u}\right) & = & \sum_{l=0}^{\Omega}\frac{1}{3^{l}}\sum_{\pi_{1},\ldots,\pi_{l-1}}\psi\left(\pi_{1}\cdots\pi_{l-1}\right)\\
 &  & \times\sum_{\pi_{l}}\mu^{2}\left(\prod_{w\neq u}D_{w}N\left(\pi_{1}\cdots\pi_{l}\right)\right)\psi\left(\pi_{l}\right)\\
\end{eqnarray*}
 where $\left(\pi_{i}\right)$ are split primes with $N\left(\pi_{i}\right)\le\Delta A_{u}$
and $A_{u}\le N\left(\pi_{l}\right)\le\Delta A_{u}/N\left(\pi_{1}\cdots\pi_{l-1}\right)$.
Then again looking at the inner sum 
\[
\sum_{\pi_{l}}\mu^{2}\left(\prod_{w\neq u}D_{w}N\left(\pi_{1}\cdots\pi_{l}\right)\right)\psi\left(\pi_{l}\right)\ll\sum_{\pi_{l},\left(\pi_{l},f_{\psi}\right)=1}\psi\left(\pi_{l}\right)+\Omega+A_{u}^{1/2}
\]
where the $\Omega$ term comes from removing $\mu^{2}$ and the $A_{u}^{1/2}$
term comes from including inert primes in the sum- recall that in
our original summation we only included primes above $l$ such that
$l\equiv1\mod3$ and $l=3$. These are exactly the split primes. The
number of inert primes $l$ in $\mathcal{O}$ with $l^{2}=Nl<A_{u}$
is $A_{u}^{1/2}$. Now we apply Proposition \ref{prop:general siegel walfisz}
with $f_{\psi}$ and $x=\Delta A_{u}/N\left(\pi_{1}\cdots\pi_{l-1}\right)$
to get 
\begin{eqnarray*}
\sum_{\pi_{l},\left(\pi_{l},f_{\psi}\right)=1}\psi\left(\pi_{l}\right) & \ll & \frac{N\left(f_{\psi}\right)^{\epsilon}x\left(\log x\right)^{2}}{\exp\left(c^{4}\left(\log x\right)^{1/2}/3^{2+\epsilon}N\left(f_{\psi}\right)^{\epsilon}\right)}\\
 & \ll & \frac{N\left(f_{\psi}\right)^{\epsilon}x\left(\log x\right)^{2}}{\exp\left(c^{4}\left(\log x\right)^{1/2}/3^{2+5\epsilon}\left(X^{\dagger}\right)^{2\cdot9^{k}\epsilon}\right)}\\
 & \ll & X^{\dagger2\cdot9^{k}\epsilon}\frac{A_{u}}{N\left(\pi_{1}\cdots\pi_{l-1}\right)}\frac{\left(\log x\right)^{2}}{\exp\left(c^{4}\left(\log X\right)^{\eta/4-2\cdot9^{k}\epsilon}/3^{2+5\epsilon}\right)}
\end{eqnarray*}
where in the second inequality we use $N\left(f_{\psi}\right)^{\epsilon}\le81\left(X^{\dagger}\right)^{2\cdot9^{k}\epsilon}$
which implies $\exp\left(-1/N\left(f_{\psi}\right)^{\epsilon}\right)\ll\exp\left(-1/81^{\epsilon}\left(X^{\dagger}\right)^{2\cdot9^{k}\epsilon}\right)$.
In the third inequality we use 
\begin{eqnarray*}
x & \ge & N\left(\pi_{l}\right)\\
 & \ge & A_{u}^{1/l}\\
 & \ge & \exp\left(\log^{\eta/2}X\right)
\end{eqnarray*}
and that $X^{\dagger}$ is some fixed power of $\log X$. Summing
over $\pi_{1}\cdots\pi_{l-1}$ gives the bound 
\[
\sum_{\pi_{1},\ldots,\pi_{l-1}}\frac{1}{N\left(\pi_{1}\cdots\pi_{l-1}\right)}\le\log A_{u}.
\]
 Finally we use $A_{v}^{\epsilon}<X^{\dagger}$ for all $v\in S$
and end up with 
\[
\ll A_{u}\frac{\left(\log X\right)^{4\left(1+9^{k}\left(1+2\cdot3^{k}\right)\right)+3}}{\exp\left(c^{4}\left(\log X\right)^{\eta/4-2\cdot9^{k}\epsilon}/3^{2+5\epsilon}\right)}
\]
 and summing over the remaining variables $D_{v}$ gives 
\[
S_{k}\left(X,\mathbf{A}\right)\ll X\frac{\left(\log X\right)^{4\left(1+9^{k}\left(1+2\cdot3^{k}\right)\right)+3}}{\exp\left(c^{4}\left(\log X\right)^{\eta/4-2\cdot9^{k}\epsilon}/3^{2+5\epsilon}\right)}.
\]
Then summing over all $\mathbf{A}$
\begin{eqnarray*}
\sum_{\mathbf{A}\mathrm{\ satisifes\ }\left(\ref{eq:family4}\right)}S_{k}\left(X,\mathbf{A}\right) & \ll & X\frac{\left(\log X\right)^{4\left(1+9^{k}\left(1+2\cdot3^{k}\right)\right)+3+9^{k}\left(1+2\cdot3^{k}\right)}}{\exp\left(c^{4}\left(\log X\right)^{\eta/4-2\cdot9^{k}\epsilon}/3^{2+5\epsilon}\right)}\\
 & = & o\left(X\right).
\end{eqnarray*}

\subsection{The third and fourth families for all $p$}

Assume GRH for Artin $L$-functions. The missing ingredients required
to extend our result to general $p$ unconditionally are analogs of
Proposition \ref{prop:bs cubic epsilon } and \ref{prop:bs dirichlet epsilon}.
That is, we cannot deal with the case in family 3 when $A_{u}$ and
$A_{v}$ are close together. We will instead give a proof assuming
GRH. The following argument replaces the sections containing families
3 and 4 for $p=3$ above.

Suppose $\mathbf{A}$ does not belong to families 1 and 2, that is
$\mathbf{A}$ does not satisfy (\ref{eq:family1}) and (\ref{eq:family2}).
In particular there are at least $p^{k-1}+1$ variables $A_{w}$ which
satisfy $A_{w}>X^{\ddagger}$. Let $A_{u}$ be the largest of these.
Let $S$ be the set of indices linked with $u$ and suppose it is
not empty.

Let $\zeta=e^{2\pi i/p}$ and let $\mathcal{O}=\mathbb{Z}\left[\zeta\right]$
the ring of integers of $\mathbb{Q}\left(\zeta\right)$ which is a
degree $p-1$ extension of $\mathbb{Q}$.

Define 
\[
\left[\frac{A}{B}\right]_{p}=\prod_{l\mid B}\frac{\left(\chi_{l}+\cdots+\chi_{l}^{p-1}\right)}{p}\left(A\right).
\]
 Then we have
\begin{equation}
S_{k}\left(X,\mathbf{A}\right)\ll\sum_{D_{w},w\neq u}\left|\sum_{D_{u}}\mu^{2}\left(D\right)\left(\frac{D_{u}}{c_{u}}\right)_{p}\left[\frac{C_{u}}{D_{u}}\right]_{p}\right|\label{eq:grhbound}
\end{equation}
where $C_{u}=\prod_{w}D_{w}^{e_{w}}$ for $e_{w}\le p-1$ and some
choice of $c_{u}\in\mathcal{O}$ with $N\left(c_{u}\right)=\prod_{w\in S}D_{w}$.
The range of summation is $A_{i}\le D_{i}\le\Delta A_{i}$ and divisibility
by $p$ is handled as in the $p=3$ cases above.

Note that 
\[
\chi:D_{u}\longmapsto\left(\frac{D_{u}}{c_{u}}\right)_{p}
\]
 is a Dirichlet character of modulus $N\left(c_{u}\right)$. Let $K=\mathbb{Q}\left(\zeta,\sqrt[p]{C_{u}}\right)$
and identify $\phi:\mathrm{Gal}\left(K/\mathbb{Q}\left(\zeta\right)\right)\longrightarrow\mu_{p}$
where $\mu_{p}$ is the group of roots of unity and define a representation
of $\mathrm{Gal}\left(K/\mathbb{Q}\left(\zeta\right)\right)$ 
\[
\tau\left(g\right)=\left(\begin{array}{ccc}
\phi\left(g\right)\\
 & \ddots\\
 &  & \phi\left(g\right)^{p-1}
\end{array}\right).
\]
 Let $\rho$ be the induction of $\tau$ to $\mathrm{Gal}\left(K/\mathbb{Q}\right)$.
Then since $D_{u}$ is a product of primes $q\equiv1\left(p\right)$
we have
\[
\left[\frac{C_{u}}{D_{u}}\right]_{p}=\frac{\prod_{q\mid D_{u}}\mathrm{tr}\rho\left(F_{q}\right)}{p^{\omega\left(D_{u}\right)}}.
\]
 Let $M$ be the degree $p$ cyclic field corresponding to the character
$\chi$ and let $L=KM$. Let $\sigma$ be the representation of $\mathrm{Gal}\left(L/\mathbb{Q}\right)$
given by 
\[
\sigma\left(q\right)=\left(\frac{q}{c_{u}}\right)_{p}\otimes\rho\left(q\right).
\]
Now consider the $L$-function
\[
L\left(s,\sigma\right)=\prod_{q}\det\left(I-\frac{\sigma\left(F_{q}\right)}{q^{-s}}\right)^{-1}
\]
where the product is over primes congruent to $1\mod3$ and not dividing
$pN\left(c_{u}\right)=p\prod_{w\in S}D_{w}$. Similarly define the
function 
\[
L\left(s\right)=\prod_{q}\left(1+\frac{\mathrm{tr}\sigma\left(F_{q}\right)}{pq^{-s}}\right).
\]
Then one can show that 
\[
L\left(s\right)=L\left(s,\sigma\right)^{1/p}F\left(s\right)
\]
where $F\left(s\right)=\prod_{q}\left(1+O\left(1/q^{2s}\right)\right)$
is absolutely convergent for $s>1/2$. 

By assumption of GRH $L\left(s,\sigma\right)$ has no zeros to the
right of $1/2$ and hence $L\left(s,\sigma\right)^{1/p}$ has no poles.
Then by a standard argument (see for instance \cite{Davenport}) we
get 
\begin{eqnarray*}
\sum_{d<x}\mu^{2}\left(pN\left(c_{u}\right)d\right)\frac{\prod_{q\mid d}\mathrm{tr}\rho\left(F_{q}\right)}{p^{\omega\left(d\right)}} & = & \int_{2-iT}^{2+iT}L\left(s\right)x^{s}\frac{ds}{s}+O\left(\frac{x^{2}}{T\log x}\right)\\
 & \ll & x^{1/2+\epsilon}\int_{1/2+\epsilon-iT}^{1/2+\epsilon+iT}\left|L\left(s,\sigma\right)^{1/p}\right|\frac{ds}{\left|s\right|}\\
 &  & +O\left(\frac{x^{2}\left(T\cdot N\left(c_{u}\right)\right)^{\epsilon}}{T}\right)+O\left(\frac{x^{2}}{T\log x}\right)\\
 & \ll & x^{1/2+\epsilon}\left(T\cdot N\left(c_{u}\right)\right)^{\epsilon}\int_{1/2+\epsilon-iT}^{1/2+\epsilon+iT}\frac{1}{\left|s\right|}ds\\
 &  & +O\left(\frac{x^{2}\left(T\cdot N\left(c_{u}\right)\right)^{\epsilon}}{T}\right)+O\left(\frac{x^{2}}{T\log x}\right)\\
 & \ll & x^{1/2+\epsilon}N\left(c_{u}\right)^{\epsilon}T^{\epsilon}+O\left(\frac{x^{2}\left(T\cdot N\left(c_{u}\right)\right)^{\epsilon}}{T}\right)+O\left(\frac{x^{2}}{T\log x}\right)\\
 & \ll & x^{1/2+\epsilon'}N\left(c_{u}\right)^{\epsilon}
\end{eqnarray*}
where the last line follows by setting $T=x^{3}$. Then we bound
the inner sum in $\left(\ref{eq:grhbound}\right)$ as 
\begin{eqnarray*}
\sum_{D_{u}}\mu^{2}\left(D\right)\left(\frac{D_{u}}{c_{u}}\right)_{p}\left[\frac{C_{u}}{D_{u}}\right]_{p} & \ll & \sum_{D_{u}}\mu^{2}\left(D\right)\frac{\prod_{q\mid d}\mathrm{tr}\rho\left(F_{q}\right)}{p^{\omega\left(D_{u}\right)}}\\
 & \ll & A_{u}^{1/2+\epsilon}N\left(c_{u}\right)^{\epsilon}.
\end{eqnarray*}
 Note that 
\begin{eqnarray*}
N\left(c_{u}\right)^{\epsilon} & = & \left(\prod_{w\in S}D_{w}\right)^{\epsilon}\\
 & \le & D_{u}^{p^{2k}\epsilon}.
\end{eqnarray*}
Then summing over all the remaining $D_{w}$ we get 
\begin{eqnarray*}
S_{k}\left(X,\mathbf{A}\right) & \le & X/A_{u}^{1/4}\\
 & = & X/X^{\ddagger1/4}\\
 & = & o\left(X\right).
\end{eqnarray*}
 This argument shows that we can remove all $\mathbf{A}$ in which
there is a variable larger than $X^{\ddagger}$ and linked with any
other $A_{w}>1$. This is equivalent to removing the $\mathbf{A}$
which belong to families 3 or 4.

We summarize the results of this section in the following theorem
\begin{thm}
\label{thm:removed four families}Let $\sum'_{\mathbf{A}}S_{k}\left(X,\mathbf{A}\right)$
denote a summation over all tuples $\mathbf{A}$ which do not belong
to any of the 4 families, that is they do not satisfy any of (\ref{eq:family1}),
(\ref{eq:family2}), (\ref{eq:family3}), (\ref{eq:family4}). Then
\[
S_{k}\left(X\right)=\sideset{}{_{\mathbf{A}}^{'}}\sum S_{k}\left(X,\mathbf{A}\right)+o\left(X\right).
\]
\end{thm}

\section{\label{sec:Computing-the--th}Computing the $k$-th Moment}

We will use the notation $\mathcal{N}\left(k\right)=\mathcal{N}\left(k,p\right)$
which we recall is the number of vector subspaces of $\mathbb{F}_{p}^{k}$.
We now want to prove Theorem \ref{thm:mainthm}. We will do this by
proving:
\begin{thm}
\label{thm: final s(x) k}
\[
S_{k}\left(X\right)=p^{-k}\left(\mathcal{N}\left(k+1\right)-\mathcal{N}\left(k\right)\right)\sum_{n<X}\left(p-1\right)^{\omega\left(n\right)-1}+o\left(X\right).
\]
\end{thm}
Note $S_{k}\left(X\right)=\sum_{K,D_{K}<X^{p-1}}\left|\mathrm{im}\left(1-\sigma_{K}\right)\right|^{k}$
is a sum over discriminants up to $X^{p-1}$. Furthermore $\sum_{n<X}\mu^{2}\left(n\right)\left(p-1\right)^{\omega\left(n\right)-1}+o\left(X\right)$
is the number of degree $p$ cyclic fields with discriminant up to
$X^{p-1}$, which is also equal to $cX+o\left(X\right)$. Thus it
follows immediately from combining these facts with the above theorem
that 

\[
\lim_{X\longrightarrow\infty}\frac{\sum_{K,D_{K}<X}\left|\mathrm{im}\left(1-\sigma_{K}\right)\right|^{k}}{\sum_{K,D_{K}<X}1}=\frac{\mathcal{N}\left(k+1\right)-\mathcal{N}\left(k\right)}{p^{k}}.
\]
We start by proving some facts about maximal unlinked sets of indices.
It will turn out that for each $\mathbf{A}$ in the sum in Theorem
\ref{thm:removed four families} all the indices $u$ with $A_{u}>1$
form a maximal unlinked set.

For $k=1$ write each index $u=u_{1}p+u_{2}\in\mathbb{\mathbb{Z}}/p^{2}\mathbb{Z}$
in base $p$ as $u=\left(u_{1},u_{2}\right)$. Recall that for $u=\left(u_{1},u_{2}\right)$
and $v=\left(v_{1},v_{2}\right)$ we defined
\[
\Phi\left(u,v\right)=\left(u_{1}\right)\left(v_{2}-u_{2}\right).
\]
 Now if we represent each index as $u=\left(u_{1},\ldots,u_{k}\right)$
as $\left(u_{11},u_{12},u_{21},u_{22},\ldots,u_{k1},u_{k2}\right)\in\mathbb{F}_{p}^{2k}$
then 
\[
\Phi_{k}\left(u,v\right)=\sum_{i=1}^{k}\left(u_{i1}\right)\left(v_{i2}-u_{i2}\right).
\]

Next we show that translating a maximal unlinked set to the origin
in $\mathbb{F}_{p}^{2k}$ yields a subspace. 
\begin{lem}
\label{lem:Let--be}Let $\mathcal{U}$ be a maximal unlinked set and
let $a\in\mathcal{U}$. Let $V=\mathcal{U}-a$. Then $V\subset\mathbb{F}_{p}^{2k}$
is a subspace. 
\end{lem}
\begin{proof}
Let $u,v\in\mathcal{U}$. We need to show that $\left(u-a\right)+\left(v-a\right)+a=u+v-a\in\mathcal{U}.$
Since $\mathcal{U}$ is maximal we show $u+v-a$ is unlinked with
every element of $\mathcal{U}$. Let $w\in\mathcal{U}$. We have
\begin{eqnarray*}
\Phi_{k}\left(u+v-a,w\right) & = & \sum_{i=1}^{k}\left(u_{i1}+v_{i1}-a_{i1}\right)\left(w_{i2}-u_{i2}-v_{i2}+a_{i2}\right)\\
 & = & \sum_{i=1}^{k}\left(u_{i1}+v_{i1}-a_{i1}\right)\left(\left(a_{i2}-u_{i2}\right)+\left(w_{i2}-v_{i2}\right)\right)\\
 & = & \sum_{i=1}^{k}v_{i1}\left(a_{i2}-u_{i2}\right)+u_{i1}\left(w_{i2}-v_{i2}\right)-a_{i1}\left(w_{i2}-v_{i2}\right)\\
 & = & 0
\end{eqnarray*}
since for instance $v_{i1}\left(-u_{i2}+a_{i2}\right)=-v_{i1}\left(u_{i2}-v_{i2}\right)+v_{i1}\left(-v_{i2}+a_{i2}\right)$
and $u,v,w,a$ are all unlinked. Similarly
\begin{eqnarray*}
\Phi_{k}\left(w,u+v-a\right) & = & \sum_{i=1}^{k}\left(w_{i1}\right)\left(u_{i2}+v_{i2}-a_{i2}-w_{i2}\right)\\
 & = & 0.
\end{eqnarray*}
This proves the lemma.
\end{proof}
It is easy to see that the $V$ in the above lemma does not depend
on the choice of $a$.

Let $p:\mathbb{F}_{p}^{2k}\longrightarrow\mathbb{F}_{p}^{k}$ be the
projection onto the even coordinates, that is 
\[
p\left(u_{11},u_{12},u_{21},u_{22},\ldots,u_{k1},u_{k2}\right)=\left(u_{12},u_{22},\ldots,u_{k2}\right)
\]
 and let $q$ be the projection onto the odd coordinates. Let $V_{1}=\ker p$
and let $V_{2}=\ker q$. Then $\mathbb{F}_{p}^{2k}=V_{1}\oplus V_{2}$.
Next we prove 
\begin{lem}
\label{lem:Let--be-1}Let $V$ be a subspace of $\mathbb{F}_{p}^{2k}$.
Then there is a maximal unlinked set $\mathcal{U}$ such that $\mathcal{U}=V+a$
if and only if $V=p\left(V\right)^{\perp}\oplus p\left(V\right)$
where $p\left(V\right)^{\perp}$ is taken in $\mathbb{F}_{p}^{k}$. 
\end{lem}
\begin{proof}
First assume $\mathcal{U}$ is unlinked, not necessarily maximal.
Suppose that $\mathcal{U}=V+a$. Then for any $w\in V$ we have $\mathcal{U}=\mathcal{U}+w$.
Let $u'\in V$ and let $u=u'+a$, so $u\in\mathcal{U}$. Note $a\in\mathcal{U}$.
Then we have
\begin{eqnarray*}
0 & = & \Phi_{k}\left(u+w,a+w\right)\\
 & = & \sum_{i=1}^{k}\left(u_{i1}+w_{i1}\right)\left(a_{i2}+w_{i2}-u_{i2}-w_{i2}\right)\\
 & = & \Phi_{k}\left(u,a\right)+\sum_{i=1}^{k}w_{i1}\left(a_{i2}-u_{i2}\right)\\
 & = & -\sum_{i=1}^{k}w_{i1}u'_{i2}.
\end{eqnarray*}
Since $w,u'\in V$ were arbitrary this shows $q\left(V\right)\in p\left(V\right)^{\perp}$.
Thus $V\subset p\left(V\right)^{\perp}\oplus p\left(V\right)$. 

Now suppose $V\subset p\left(V\right)^{\perp}\oplus p\left(V\right)$.
Let $a=0$. Let $\mathcal{U}=V$. For $v,w\in V$ 
\begin{eqnarray*}
\Phi_{k}\left(v,w\right) & = & \sum_{i=1}^{k}v_{i1}\left(w_{i2}-v_{i2}\right)\\
 & = & 0
\end{eqnarray*}
 by the assumption and since $w-v\in V$. Hence $\mathcal{U}$ is
unlinked.

Now if $V$ satisfies the above then clearly $p\left(V\right)^{\perp}\oplus p\left(V\right)$
also does so if $\mathcal{U}$ is maximal then we have equality $V=p\left(V\right)^{\perp}\oplus p\left(V\right)$.

Conversely note that equality $V=p\left(V\right)^{\perp}\oplus p\left(V\right)$
implies $\dim V=k$. If $\mathcal{U}$ is not maximal then let $\mathcal{U}'$
be a maximal unlinked set containing it. By Lemma \ref{lem:Let--be}
and the first part there is a subspace $V'=\mathcal{U}'-a$ which
will contain $V$ and such that $V'=p\left(V'\right)^{\perp}\oplus p\left(V'\right)$.
Hence $\dim V'=k$ also and we must have $V'=V$ so $\mathcal{U}=\mathcal{U}'$
is maximal.
\end{proof}
Hence the maximal unlinked sets are all obtained as translates of
subspaces which satisfy the conditions of Lemma \ref{lem:Let--be-1}.
Note for such subspaces $V$ that $\dim V=\dim W\oplus p\left(V\right)=k$
and hence every maximal unlinked set has size $p^{k}$. 

With this we can rewrite $S_{k}\left(X\right)$ in a form closer to
Theorem \ref{thm: final s(x) k}.
\begin{prop}
Let $U$ be the number of maximal unlinked sets. Then

\[
\sum_{K,D_{K}<X^{p-1}}\left|\mathrm{im}\left(1-\sigma_{K}\right)\right|^{k}=\left(\frac{U}{p^{k}}\right)\sum_{n<X}\left(p-1\right)^{\omega\left(n\right)-1}+o\left(X\right)
\]
\end{prop}
\begin{proof}
From the work above it is easy to show that the largest possible intersection
of two maximal unlinked sets has size $p^{k-1}$. Hence a set of $p^{k-1}+1$
unlinked variables determines a unique maximal unlinked set. Thus
each family $\mathbf{A}$ in Theorem \ref{thm:removed four families}
corresponds to a unique maximal unlinked set. We can partition $\sum'{}_{\mathbf{A}}S\left(X,\mathbf{A}\right)$
according to these maximal unlinked sets, and additionally adding
back the previous error term consisting of sums $S\left(X,\mathbf{A}\right)$
which have at most $p^{k-1}$ variables with $A_{i}>X^{\ddagger}$
at the cost of an error $X/\log X$ we get

\begin{eqnarray*}
\sideset{}{_{\mathbf{A}}^{'}}\sum S\left(X,\mathbf{A}\right) & = & \frac{U}{p^{k}\left(p-1\right)}\sum_{0\le\prod_{j=0}^{p^{k}}n_{j}\le X}\mu^{2}\left(\prod_{j=0}^{p^{k}}n_{j}\right)\left(\frac{p-1}{p}\right)^{\omega\left(\prod_{j=0}^{p^{k}}n_{j}\right)}+o\left(X\right)\\
 & = & \frac{U}{p^{k}}\sum_{n<X}\mu^{2}\left(n\right)\left(p-1\right)^{\omega\left(n\right)-1}+o\left(X\right).
\end{eqnarray*}
\end{proof}
The final step of the proof will be the next proposition. Define $n\left(k,r\right)$
to be the number of $r$-dimensional subspaces of $\mathbb{F}_{p}^{k}$.
We will need two properties of this function which can be found in
Lemmas 1 and 3 from \cite{fk1}:
\begin{lem}
\label{lem:The-function-}The function $n\left(k,r\right)$ satisfies
\[
n\left(k,r\right)=n\left(k,k-r\right),
\]
\[
\sum_{r=0}^{k}p^{r}n\left(k,r\right)=\mathcal{N}\left(k+1\right)-\mathcal{N}\left(k\right).
\]
\end{lem}
\begin{prop}
The number of maximal unlinked sets $U$ is
\[
U=\mathcal{N}\left(k+1\right)-\mathcal{N}\left(k\right).
\]
 
\end{prop}
\begin{proof}
By an argument similar to the proof of Lemma \ref{lem:Let--be-1}
we can show that if $\mathcal{U}$ is maximal unlinked and $\mathcal{U}=V+a'$
then $\mathcal{U}+a$ is maximal unlinked if and only if $q\left(a\right)\in p\left(V\right)^{\perp}$.
Hence given any $k$-dimensional subspace $V\subset\mathbb{F}_{p}^{2k}$
there are 
\[
p^{k}\left(p^{\dim p\left(V\right)^{\perp}}\right)
\]
 vectors which translate $V$ to a maximal unlinked set. However since
translating by $a_{1}$ and $a_{2}$ gives the same set if and only
if $a_{1}$ and $a_{2}$ are in the same coset of $V$ this implies
that there are $\left(p^{\dim p\left(V\right)^{\perp}}\right)$ distinct
maximal unlinked sets that can be obtained from $V$. 

Now let $S$ be the set of $k$-dimensional subspaces $V\subset\mathbb{F}_{p}^{2k}$
which satisfy Lemma \ref{lem:Let--be-1}. We compute the size of this
set. Fix some subspace $V_{0}\subset\mathbb{F}_{p}^{k}$ with $\dim V_{0}=r$
and suppose $V$ satisfies $p\left(V\right)=V_{0}$. So $\dim p\left(V\right)^{\perp}=k-r$.
We want $V\subset p\left(V\right)^{\perp}\oplus p\left(V\right)=V_{0}^{\perp}\oplus V_{0}$
but both sides have dimension $k$ and hence there is a unique $V$
with $p\left(V\right)=V_{0}$ and satisfying the condition of Lemma
\ref{lem:Let--be-1}. Hence the number of $V\in\left|S\right|$ with
$\dim p\left(V\right)=r$ is $n\left(k,r\right)$.

Thus we have
\begin{eqnarray*}
U & = & \sum_{V\in S}p^{\dim p\left(V\right)^{\perp}}\\
 & = & \sum_{r=0}^{k}p^{r}n\left(k,r\right)\\
 & = & \mathcal{N}\left(k+1\right)-\mathcal{N}\left(k\right)
\end{eqnarray*}
by Lemma \ref{lem:The-function-}.
\end{proof}
Thus we have computed 
\[
S_{k}\left(X\right)=p^{-k}\left(\mathcal{N}\left(k+1\right)-\mathcal{N}\left(k\right)\right)\sum_{n<X}\left(p-1\right)^{\omega\left(n\right)-1}+o\left(X\right)
\]
which proves Theorem \ref{thm: final s(x) k}. As remarked at the
beginning of the section it follows that 
\[
\lim_{X\longrightarrow\infty}\frac{\sum_{K.D_{K}<X}\left|\mathrm{im}\left(1-\sigma_{K}\right)\right|^{k}}{\sum_{K,D_{K}<X}1}=p^{-k}\left(\mathcal{N}\left(k+1\right)-\mathcal{N}\left(k\right)\right).
\]
By the results of \cite{fk2distribution} Theorems \ref{thm:gerth 3 rank}
and \ref{thm:gerth p ranks} follow immediately.

\bibliographystyle{plain}
\bibliography{bibliography}

\textsc{~}

\textsc{\small{}Department of Mathematics, University of Toronto,
Canada}{\small \par}

{\small{}$\mathtt{jack.klys@mail.utoronto.ca}$}{\small \par}
\end{document}